\documentclass[secthm,seceqn,amsthm,ussrhead,reqno, 12pt]{amsart}
\usepackage[utf8]{inputenc}
\usepackage[english]{babel}
\usepackage{amssymb,amsmath,amsthm,amsfonts,xcolor,enumerate,hyperref,comment,longtable,cleveref}
\usepackage[symbol]{footmisc}
\usepackage{times}
\usepackage{cite}
\usepackage{pdflscape}
\usepackage{ulem}
\usepackage[mathcal]{euscript}
\usepackage{tikz}
\usepackage{hyperref}
\usepackage{cancel}
\usepackage{stmaryrd}

\newtheorem{thm}{Theorem}[section]
\newtheorem{cor}[thm]{Corollary}

\newtheorem{prop}[thm]{Proposition}
\newtheorem{definition}[thm]{Definition}

\numberwithin{equation}{section}
\DeclareFontFamily{OT1}{pzc}{}
\DeclareFontShape{OT1}{pzc}{m}{it}{<-> s * [1.10] pzcmi7t}{}
\DeclareMathAlphabet{\mathpzc}{OT1}{pzc}{m}{it}

 \DeclareMathOperator{\gr}{gr}

\begin{document}
\title[Local and 2-local derivations on filiform associative algebras]{Local and 2-local derivations on filiform associative algebras}

\author[ Abdurasulov K.K.,  Ayupov Sh.A., Yusupov B.B. ]{Kobiljon Abdurasulov$^1$, Shavkat Ayupov$^{1,2}$,  Bakhtiyor Yusupov$^{1,3}$}
\address{$^1$ V.I.Romanovskiy Institute of Mathematics\\
  Uzbekistan Academy of Sciences, 9 \\ Univesity street, 100174  \\
  Tashkent,   Uzbekistan}

  \address{$^2$ National University of Uzbekistan, \\Tashkent 100170, Uzbekistan}

\address{$^3$ Department of Physics and Mathematics, Urgench State University, H. Alimdjan street, 14, Urgench
220100, Uzbekistan}
\email{\textcolor[rgb]{0.00,0.00,0.84}{abdurasulov0505@mail.ru,  sh\_ayupov@mail.ru, baxtiyor\_yusupov\_93@mail.ru}}
\maketitle

\begin{abstract} This paper is devoted to the study of local and 2-local derivations  of null-filiform, filiform and naturally graded
quasi-filiform associative algebras. We prove that these algebras as a rule admit local derivations which are not derivations. We show that filiform and naturally graded quasi-filiform associative algebras admit  2-local derivations which are not  derivations and any 2-local derivation of null-filiform associative algebras is a derivation.
\end{abstract}
\\


\section{Introduction}

R.Kadison \cite{Kad} introduced the concept of a local derivation and proved
that each continuous local derivation from a von Neumann algebra into its dual Banach
bemodule is a derivation. B. Jonson \cite{Jon} extended the above result by proving that every
local derivation from a C*-algebra into its Banach bimodule is a derivation. In particular, Johnson
gave an automatic continuity result by proving that local derivations of a C*-algebra $A$ into a
Banach $A$-bimodule $X$ are continuous even if not assumed a priori to be so
(cf. \cite[Theorem 7.5]{Jon}). Based on these results, many authors have studied
local derivations on operator algebras.

A similar notion, which characterizes non-linear generalizations of automorphisms and derivations, was introduced
by  P.\v{S}emrl in \cite{S} as $2$-local automorphisms (respectively, 2-local derivations).
He described such maps on the algebra $B(H)$ of all bounded linear operators on an infinite dimensional separable Hilbert space $H$.
After P.\v{S}emrl's work, numerous new results related to the description of local and $2$-local derivation of some varieties of algebras have been appeared
(see, for example, \cite{AKO,WCN1,ChWD}).

The first results concerning local and 2-local derivations and automorphisms on finite-dimensional Lie algebras over algebraically closed
field of zero characteristic were obtained in \cite{AK,AKR}. Namely, in \cite{AKR} it is proved
that every 2-local derivation on a semi-simple Lie algebra $\mathcal{L}$ is a derivation and that
each finite-dimensional nilpotent Lie algebra with dimension larger than two admits
2-local derivation which is not a derivation. In \cite{AK} the authors have proved that every local
derivation on semi-simple Lie algebras is a derivation and gave examples of nilpotent
finite-dimensional Lie algebras with local derivations which are not derivations.
 Sh.Ayupov, K.Kudaybergenov, B.Omirov proved similar results concerning local and 2-local derivations and automorphisms on simple Leibniz algebras  in their recent paper \cite{AKO}. In \cite{AyuKhuYus} local derivations of solvable Leibniz algebras are investigated and it is shown that in the class of solvable Leibniz algebras there exist algebras which admit local derivations which are not derivation and also algebras for which every local derivation
is a derivation. Moreover, it is proved that every local derivation on a finite-dimensional solvable Leibniz algebras with model
nilradical and maximal dimension of complementary space is a derivation. The results of the paper \cite{AyuKudYus1} show that p-filiform Leibniz algebras
as a rule admit local derivations which are not derivations. J.Adashev and B.Yusupov proved similar results concerning local derivations of naturally graded quasi-filiform Leibniz algebras in their recent paper \cite{AdaYus}. J.Adashev and B.Yusupov proved proved that direct sum of null-filiform nilpotent Leibniz algebras as a rule admit local derivations which are not derivations \cite{AdaYus1}.
The first example of a simple (ternary) algebra with nontrivial local derivations is constructed by B.Ferreira, I.Kaygorodov and K.Kudaybergenov in \cite{FKK}. After that, in \cite{AEK1} Sh.Ayupov, A.Elduque and K.Kudaybergenov constructed an example for a simple (binary) algebra with non-trivial local derivations.

In the present paper we study derivations, local derivations and 2-local derivations on filiform associative algebras. In Section 3 we describe the derivations of null-filiform, filiform and naturally graded quasi-filiform associative algebras. In Section 4  we consider local and 2-local derivations of small dimensional associative algebras. We show that small dimensional associative algebras admit a local (2-local) derivations which is not a derivation. In Section 5 we consider  local and 2-local derivations  on arbitrary finite-dimensional null-filiform, filiform and naturally graded quasi-filiform associative algebras. In subsection 5.1 we show that these algebras as a rule admit local derivations which are not derivations. In subsection 5.2   we show that filiform and naturally graded quasi-filiform associative algebras admit 2-local derivations which are not derivations and any 2-local derivation of null-filiform associative algebras is a derivation.

\section{Preliminaries}\label{S:prel}

For an algebra $\bf {A}$ of an arbitrary variety, we consider the series
\[
{\bf A}^1={\bf A}, \qquad \ {\bf A}^{i+1}=\sum\limits_{k=1}^{i}{\bf A}^k {\bf A}^{i+1-k}, \qquad i\geq 1.
\]

We say that  an  algebra $\bf A$ is nilpotent if ${\bf  A}^{i}=0$ for some $i \in \mathbb{N}$. The smallest integer satisfying ${\bf A}^{i}=0$ is called the  index of nilpotency of or nilindex $\bf A$.

\begin{definition}
An $n$-dimensional algebra $\bf A$ is called null-filiform if $\dim {\bf A}^i=(n+ 1)-i,\ 1\leq i\leq n+1$.
\end{definition}

It is easy to see that an algebra has a maximum nilpotency index if and only if it is null-filiform. For a nilpotent algebra, the condition of null-filiformity is equivalent to the condition that the algebra is one-generated.

All null-filiform associative algebras were described in  \cite{MO,karel}:

\begin{thm} An arbitrary $n$-dimensional null-filiform associative algebra is isomorphic to the algebra:
\[\mu_0^n : \quad e_i e_j= e_{i+j}, \quad 2\leq i+j\leq n,\]
where $\{ e_1, e_2, \dots, e_n\}$ is a basis of the algebra $\mu_0^n$.
\end{thm}

The following classes of nilpotent algebras are filiform and quasi-filiform algebras whose nilindex equals $n$ and $n-1$, respectively.

\begin{definition}
An $n$-dimensional algebra is called filiform if $\dim (\mathbf{A} ^i)=n-i, \ 2\leq i \leq n$.
\end{definition}

\begin{definition}
An $n$-dimensional associative algebra $\bf{A}$ is called quasi-filiform algebra if ${\bf A}^{n-2}\neq0$ and ${\bf A}^{n-1}=0$.
\end{definition}

\begin{definition}
Given a nilpotent associative algebra $\bf{A}$, put
${\bf A}_i={\bf A}^i/{\bf A}^{i+1}, \ 1 \leq i\leq k-1$, and
$\gr {\bf A} = {\bf A}_1 \oplus {\bf A}_2\oplus\dots \oplus {\bf A}_{k}$.
Then ${\bf A}_i{\bf A}_j\subseteq {\bf A}_{i+j}$ and we
obtain the graded algebra $\gr \bf{A}$. If the algebras $\gr \bf{A}$ and $\bf{A}$ are isomorphic,
denoted by $\gr {\bf A} \cong {\bf A}$, then  we say that the algebra $\bf{A}$ is naturally
graded.

\end{definition}

All filiform and naturally graded quasi-filiform associative algebras were classified in  \cite{KL}.

\begin{thm}[\cite{KL}]
Every $n$-dimensional ($n>3$) complex filiform associative algebra is isomorphic to one of the following pairwise non-isomorphic algebras with basis $\{ e_1,e_2,\dots,e_n\}$:
\[\begin{array}{lllll}
\mu_{1,1}^n(\mu_0^{n-1}\oplus\mathbb{C}) &: & e_ie_j=e_{i+j},  & & \\
\mu_{1,2}^n &: & e_ie_j=e_{i+j},  & e_ne_n=e_{n-1}, & \\
\mu_{1,3}^n&: & e_ie_j=e_{i+j},  & e_1e_n=e_{n-1}, & \\
\mu_{1,4}^n&: & e_ie_j=e_{i+j},  & e_1e_n=e_{n-1}, & e_ne_n=e_{n-1}
\end{array}\]
where $2\leq i+j\leq n-1$.
\end{thm}

\begin{thm}[\cite{KL}]
 Let $\bf{A}$ be $n$-dimensional $(n \geq 6)$ complex naturally graded quasi-filiform  non-split associative algebra. Then it is isomorphic to one of the following pairwise non-isomorphic algebras:
\[\begin{array}{llllll}
\mu_{2,1}^n&: &  e_ie_j=e_{i+j},  &   e_{n-1}e_1=e_n   \\
\mu_{2,2}^n(\alpha)&: & e_ie_j=e_{i+j}, & e_1e_{n-1}=e_n, & e_{n-1}e_1=\alpha e_n \\
\mu_{2,3}^n&: & e_ie_j=e_{i+j}, & e_{n-1}e_{n-1}=e_n & \\
\mu_{2,4}^n&: & e_ie_j=e_{i+j}, & e_1e_{n-1}=e_n, & e_{n-1}e_{n-1}=e_n\\
\end{array}\]
where $\alpha\in\mathbb{C}$ and $2\leq i+j\leq n-2$.
\end{thm}

\begin{thm}\label{dim2filiform}\cite{Will} Any 3-dimensional complex nilpotent associative algebra $\mathcal{A}$ with a nonzero product
is isomorphic to one of the following pairwise non-isomorphic algebras with a basis $\{e_1,e_2,e_3\}$:
\begin{equation*}\begin{split}
\mathcal{A}_1:& [e_1,e_1]=e_2,\\
\mathcal{A}_2:& [e_1,e_2]=[e_2,e_1]=e_3,\\
\mathcal{A}_3:& [e_1,e_2]=-[e_2,e_1]=e_3,\\
\mathcal{A}_4^{\alpha}:& [e_1,e_1]=e_3,\ [e_2,e_2]=\alpha e_3,\ [e_1,e_2]=e_3,\quad \alpha\in\mathbb{C},\\
\mathcal{A}_5: & [e_1,e_1]=e_2,\ [e_1,e_2]=[e_2,e_1]=e_3.\
\end{split}\end{equation*}
Here the multiplication is specified by giving only the nonzero products and for different values of $\alpha\,$ the obtained algebras are non-isomorphic.
\end{thm}


For a given $n$ - dimensional nilpotent associative algebra $A$ we define the
following isomorphism invariant:
$$\chi(A)=(dimA, dimA^2, dimA^3,\dots, dimA^n).$$
It is clear that
$$dimA > dimA^2 > dimA^3 > \dots > dimA^n.$$

\begin{prop}\cite{Kar}
Any $5$ - dimensional nilpotent complex associative algebra $A$ with $A^4=0$ and  $A^3\neq0$
belongs to one of the following two types of algebras:
$$\chi(A)=(5,2,1,0,0), \quad  \chi(\mathcal{A})=(5,3,1,0,0).$$
\end{prop}

Now we give a classification of 5 – dimensional nilpotent complex associative algebras $\mathcal{A}$ which
satisfy $\chi(A)=(5,2,1,0,0)$.
\begin{thm}\cite{Kar}
Let A be a 5 - dimensional nilpotent complex non-split associative algebra
with $\chi(\mathcal{A})=(5,2,1,0,0).$ Then $\mathcal{A}$ is isomorphic to one of the following pairwise non-isomorphic associative algebras spanned by $\{e_1, e_2, e_3, e_4, e_5 \}$ with the nonzero products given by:
\[\begin{array}{lll}
    \lambda_1:\left\{\begin{array}{l}
e_1e_1=e_2,\\
e_1e_2=e_2e_1=e_3,\\
e_4e_4=e_3,\\
e_5e_5=e_3\end{array}\right. & \lambda_2:\left\{\begin{array}{l}
e_1e_1=e_2,\\
e_1e_2=e_2e_1=e_3,\\
e_1e_4=e_3,\\
e_4e_5=e_5e_4=e_3
\end{array}\right. & \lambda_3: \left\{\begin{array}{l}
e_1e_1=e_2, \\
e_1e_2=e_2e_1=e_3,\\
e_1e_4=e_3,\\
e_5e_5=e_3
\end{array}\right. \\
    \lambda_4:\left\{\begin{array}{l}
e_1e_1=e_2,\\
e_1e_2=e_2e_1=e_3,\\
e_1e_4=e_3,\\
e_4e_4=e_5e_5=e_3\end{array}\right. &  \lambda_5: \left\{\begin{array}{l}
e_1e_1=e_2, \\
e_1e_2=e_2e_1=e_3,\\
e_4e_5=-e_5e_4=e_3
\end{array}\right.&
\lambda_6^\alpha: \left\{\begin{array}{l}
e_1e_1=e_2, \\
e_1e_2=e_2e_1=e_3,\\
e_4e_4=e_4e_5=e_3,\\
e_5e_5=\alpha e_3
\end{array}\right.
  \end{array}\]

where $\alpha\in\mathbb{C}$ and for different values of $\alpha$, the obtained algebras are non-isomorphic.
\end{thm}

Finally, we give a classification of associative algebras with $\chi(\mathcal{A})=(5,3,1,0,0)$.

\begin{thm}\cite{Kar}
Let $\mathcal{A}$ be a 5 - dimensional complex non-split nilpotent associative algebra
with $\chi(\mathcal{A})=(5,3,1,0,0).$ Then $\mathcal{A}$ is isomorphic to one of the following pairwise non-isomorphic associative algebras spanned by $\{ e_1, e_2, e_3, e_4, e_5 \}$ with the nonzero products given by:
\[\mu_1:\left\{\begin{array}{l}
e_1e_1=e_2,\\
e_1e_2=e_2e_1=e_3,\\
e_4e_1=e_5\end{array}\right.
\quad
\mu_2:\left\{\begin{array}{l}
e_1e_1=e_2,\\
e_1e_2=e_2e_1=e_3,\\
e_4e_1=e_5,\\
e_4e_4=e_3\end{array}\right.
\quad
\mu_3:\left\{\begin{array}{l}
e_1e_1=e_2,\\
e_1e_2=e_2e_1=e_3,\\
e_4e_1=e_5,\\
e_4e_2=e_5e_1=e_3
\end{array}\right.\]
\[\mu_4:\left\{\begin{array}{l}
e_1e_1=e_2,\\
e_1e_2=e_2e_1=e_3,\\
e_4e_1=e_5,\\
e_4e_2=e_5e_1=e_3,\\
e_4e_4=e_3
\end{array}\right.
\quad
\mu_5:\left\{\begin{array}{l}
e_1e_1=e_2,\\
e_1e_2=e_2e_1=e_3,\\
e_1e_4=e_5,\\
e_4e_1=e_3+e_5,
\end{array}\right.
\quad
\mu_6:\left\{\begin{array}{l}
e_1e_1=e_2,\\
e_1e_2=e_2e_1=e_3,\\
e_1e_4=e_5,\\
e_4e_1=e_3+e_5,\\
e_4e_4=e_3
\end{array}\right.\]
\[\mu_7^\alpha:\left\{\begin{array}{l}
e_1e_1=e_2,\\
e_1e_2=e_2e_1=e_3,\\
e_1e_4=e_5,\\
e_4e_1=\alpha e_5,
\end{array}\right.
\quad
\mu_8^\alpha:\left\{\begin{array}{l}
e_1e_1=e_2,\\
e_1e_2=e_2e_1=e_3,\\
e_1e_4=e_5,\\
e_4e_1=\alpha e_5,\\
e_4e_4=e_3
\end{array}\right.
\quad
\mu_9:\left\{\begin{array}{l}
e_1e_1=e_2,\\
e_1e_2=e_2e_1=e_3,\\
e_4e_1=e_3,\\
e_4e_4=e_5
\end{array}\right.\]
\[
\mu_{10}: \left\{\begin{array}{l}
e_1e_1=e_2, \\
e_1e_2=e_2e_1=e_3,\\
e_1e_4=e_5,\\
e_4e_4=e_5
\end{array}\right.
\quad
\mu_{11}: \left\{\begin{array}{l}
e_1e_1=e_2, \\
e_1e_2=e_2e_1=e_3,\\
e_1e_4=e_5,\\
e_4e_4=e_3+e_5
\end{array}\right.
\quad
\mu_{12}:\left\{\begin{array}{l}
e_1e_1=e_2, \\
e_1e_2=e_2e_1=e_3,\\
e_1e_4=e_5,\\
e_4e_1=e_2-e_5,\\
e_5e_1=e_3
\end{array}\right.
\]
\[\mu_{13}:\left\{\begin{array}{l}
e_1e_1=e_2, \\
e_1e_2=e_2e_1=e_3,\\
e_1e_4=e_5,\\
e_4e_1=e_2-e_5,\\
e_4e_4=e_3,\\
e_5e_1=e_3
\end{array}\right.
\quad
\mu_{14}:\left\{\begin{array}{l}
e_1e_1=e_2,\\
e_1e_2=e_2e_1=e_3,\\
e_1e_4=e_5,\\
e_4e_1=e_2+e_5,\\
e_4e_2=2e_3,\\
e_4e_4=2e_5, \\
e_5e_1=e_3
\end{array}\right.
\quad
\mu_{15}:\left\{\begin{array}{l}
e_1e_1=e_2,\\
e_1e_2=e_2e_1=e_3,\\
e_1e_4=e_5,\\
e_4e_1=e_2+e_5,\\
e_4e_2=2e_3,\\
e_4e_4=e_3+2e_5, \\
e_5e_1=e_3
\end{array}\right.
\]
\[\mu_{16}:\left\{\begin{array}{l}
e_1e_1=e_2,\\
e_1e_2=e_2e_1=e_3,\\
e_1e_4=e_4e_1=e_5,\\
e_4e_4=e_2, \\
e_4e_5=e_5e_4=e_3
\end{array}\right.
\quad
\mu_{17}:\left\{\begin{array}{l}
e_1e_1=e_2,\\
e_1e_2=e_2e_1=e_3,\\
e_1e_4=e_5,\\
e_4e_1=e_3+e_5,\\
e_4e_4=e_2, \\
e_4e_5=e_5e_4=e_3
\end{array}\right.
\quad
\mu_{18}:\left\{\begin{array}{l}
e_1e_1=e_2, \\
e_1e_2=e_2e_1=e_3,\\
 e_1e_4=-e_4e_1=e_5,\\
e_4e_4=e_2,\\
e_5e_4=-e_4e_5=e_3
\end{array}\right.\]
\[
\mu_{19}:\left\{\begin{array}{l}
e_1e_1=e_2, \\
e_1e_2=e_2e_1=e_3,\\
e_1e_4=e_5,\\
e_4e_1=e_2,\\
e_4e_2=e_3, \\
e_4e_4=e_3+e_5,\\
e_5e_1=e_3
\end{array}\right.
\quad
\mu_{20}:\left\{\begin{array}{l}
e_1e_1=e_2, \\
e_1e_2=e_2e_1=e_3,\\
e_1e_4=e_5,\\
e_4e_1=e_3+e_5,\\
e_4e_4=-e_2+2e_5,\\
e_4e_5=e_5e_4=-e_3
\end{array}\right.
\]
\[
\mu_{21}^\alpha:\left\{\begin{array}{ll}
e_1e_1=e_2, &
e_1e_2=e_2e_1=e_3,\\
e_1e_4=e_5, &
e_4e_1=(1-\alpha)e_2+\alpha e_5,\\
e_4e_2=2e_3, &
e_4e_4=-\alpha e_2+e_3+(1+\alpha)e_5,\\
e_4e_5=e_3,&
e_5e_1=(1-\alpha)e_3,\\
e_5e_4=-\alpha e_3, & \alpha\in\{\pm\mathbf{i}\}.
\end{array}\right.\] \[\mu_{22}^\alpha:\left\{\begin{array}{ll}
e_1e_1=e_2, & e_1e_2=e_2e_1=e_3,\\
e_1e_4=e_5, &  e_4e_1=(1-\alpha)e_2+\alpha e_5,  \\
e_4e_2=(1-\alpha^2)e_3, &  e_4e_4=-\alpha e_2+(1+\alpha)e_5,   \\
e_4e_5=-\alpha^2e_3, & e_5e_1=(1-\alpha)e_3, \\
e_5e_4=-\alpha e_3, & \alpha\in\mathbb{C},
\end{array}\right.
\]
for different values of $\alpha$, the obtained algebras are non-isomorphic.
\end{thm}
\begin{definition}
A derivation on an algebra $\mathcal{A}$ is a linear map  $D: \mathcal{A} \rightarrow \mathcal{A}$  which satisfies the Leibniz rule:
\begin{equation}\label{der}
D(x,y) = D(x), y + x, D(y) \quad \text{for any} \quad x,y \in \mathcal{A}.
\end{equation}
\end{definition}

The set of all derivations of a algebra $\mathcal{A}$ is a Lie algebra with respect to
commutation operation and it is denoted by $Der(\mathcal{A}).$


\begin{definition}
A linear operator $\Delta$ is called a local derivation if for any $x \in \mathcal{A},$ there exists a derivation $D_x: \mathcal{A} \rightarrow \mathcal{A}$ (depending on $x$) such that
$\Delta(x) = D_x(x).$
\end{definition}
The set of all local derivations on $\mathcal{A}$ we denote by $LocDer(\mathcal{A}).$
\begin{definition}
A map $\nabla :  \mathcal{A} \rightarrow \mathcal{A}$ (not
necessary  linear) is called $2$-local derivation if for any $x,y\in \mathcal{A}$ there exists a derivation $D_{x,y}\in Der
(\mathcal{A})$ such that
\[
\nabla(x)=D_{x,y}(x), \quad \nabla(y)=D_{x,y}(y).
\]
\end{definition}

The set of all 2-local derivations on $\mathcal{A}$ we denote by $TLocDer(\mathcal{A}).$

For a 2-local derivation $\nabla$ on $\mathcal{A}$ and $k\in \mathbb{C}$, $x\in \mathcal{A}$, we have
\begin{equation*}
\nabla(kx)=D_{x,kx}(kx)=kD_{x,kx}(x)=k\nabla(x).
\end{equation*}

\section{Derivations on null-filiform and filiform associative algebras}

Now we study derivations on $n$-dimensional null-filiform, $n$-dimensional ($n>2$) complex filiform and $n$-dimensional $(n \geq 5)$ complex naturally graded quasi-filiform  non-split associative algebras.

\begin{prop}\label{prop}
Derivations of  $n$-dimensional complex null-filiform associative algebras are given as follows:
$$\begin{array}{llll}
D(e_{i})&=&i\sum\limits_{k=i}^{n}\alpha_{k-i+1}e_{k}, & 1\leq i\leq n,\\
\end{array}$$
\end{prop}
\begin{proof} Since the algebra $\mu_{0}^n$ has one generator $\{e_{1}\}$, any derivation $D$ on $\mu_{0}^n$  is completely determined by $D(e_{1})$.

Let
$$D(e_{1})=\sum_{i=1}^{n}\alpha_{i}e_{i}.$$

Applying the derivation rule we have
$$D(e_2)=D(e_1e_1)=D(e_1)e_1+e_1D(e_1)=\sum_{i=1}^{n}\alpha_{i}e_{i}e_1+\sum_{i=1}^{n}\alpha_{i}e_1e_{i}=$$
$$=\sum_{i=2}^{n}\alpha_{i-1}e_{i}+\sum_{i=2}^{n}\alpha_{i-1}e_{i}=2\sum_{i=2}^{n}\alpha_{i-1}e_{i}.$$

$$D(e_3)=D(e_2e_1)=D(e_2)e_1+e_2D(e_1)=2\sum_{i=2}^{n}\alpha_{i-1}e_{i}e_1+\sum_{i=1}^{n}\alpha_{i}e_2e_{i}=$$
$$=2\sum_{i=3}^{n}\alpha_{i-2}e_{i}+\sum_{i=3}^{n}\alpha_{i-2}e_{i}=3\sum_{i=3}^{n}\alpha_{i-2}e_{i}.$$

Now by induction we shall show that
$$D(e_i)=i\sum_{k=i}^{n}\alpha_{k-i+1}e_{k}, \quad \quad 2\leq i\leq n.$$

Suppose the above is true for some $i$. Now we can determine for $i + 1$.
$$D(e_{i+1})=D(e_ie_1)=D(e_i)e_1+e_iD(e_1)=i\sum_{k=i}^{n}\alpha_{k-i+1}e_{k}e_1+\sum_{k=1}^{n}\alpha_{k}e_ie_{k}=$$
$$=i\sum_{k=i+1}^{n}\alpha_{k-i}e_{k}+\sum_{k=i+1}^{n}\alpha_{k-i}e_{k}=(i+1)\sum_{k=i+1}^{n}\alpha_{k-i}e_{k}.$$

Let us take $e_i,e_j\in \mu_0^{n}$ and denote them by 
$$D(e_i)=i\sum_{k=i}^{n}\alpha_{k-i+1}e_{k},\quad D(e_j)=j\sum_{k=j}^{n}\alpha_{k-j+1}e_{k}, \quad 1\leq i\leq n.$$

Consider 
$$D(e_ie_j)=D(e_i)e_j+e_iD(e_j)=i\sum_{k=i}^{n}\alpha_{k-i+1}e_{k}e_j+j\sum_{k=j}^{n}\alpha_{k-j+1}e_ie_{k}$$
$$=i\sum_{k=i}^{n}\alpha_{k-i+1}e_{k+j}+j\sum_{k=j}^{n}\alpha_{k-j+1}e_{i+k}=(i+j)\sum_{k=i+j}^{n}\alpha_{k-i-j+1}e_{k}.$$

On the other hand, 
$$D(e_ie_j)=D(e_{i+j})=(i+j)\sum_{k=i+j}^{n}\alpha_{k-i-j+1}e_{k}.$$
We complete the proof of theorem.
\end{proof}

\begin{prop}\label{prop1}
Derivations of the $n$-dimensional ($n>3$) complex filiform associative algebras are given as follows:

\begin{itemize}
  \item for the algebra $\mu_{1,1}^n$
$$\begin{array}{llll}
D(e_{1})&=&\sum\limits_{i=1}^{n}\alpha_{i}e_{i},\\
D(e_{i})&=&i\sum\limits_{k=i}^{n-1}\alpha_{k-i+1}e_{k}, & 2\leq i\leq n-1,\\
D(e_{n})&=&\beta_{n-1}e_{n-1}+\beta_{n}e_{n}.
\end{array}$$
\item for the algebra $\mu_{1,2}^n$
$$\begin{array}{llll}
D(e_{1})&=&\sum\limits_{i=1}^{n}\alpha_{i}e_{i},\\
D(e_{i})&=&i\sum\limits_{k=i}^{n-1}\alpha_{k-i+1}e_{k}, & 2\leq i\leq n-1,\\
D(e_{n})&=&\alpha_ne_{n-2}+\beta_{n-1}e_{n-1}+\frac{n-1}2\alpha_1e_{n}.
\end{array}$$
\item for the algebra $\mu_{1,3}^n$
$$\begin{array}{llll}
D(e_{1})&=&\sum\limits_{i=1}^{n}\alpha_{i}e_{i},\\
D(e_{2})&=&2\sum\limits_{k=2}^{n-2}\alpha_{k-1}e_{k}+(2\alpha_{n-2}+\alpha_n)e_{n-1}, & \\
D(e_{i})&=&i\sum\limits_{k=i}^{n-1}\alpha_{k-i+1}e_{k}, & 3\leq i\leq n-1,\\
D(e_{n})&=&\beta_{n-1}e_{n-1}+(n-2)\alpha_1e_{n}.
\end{array}$$
\item for the algebra $\mu_{1,4}^n$
$$\begin{array}{llll}
D(e_{1})&=&\sum\limits_{i=1}^{n-1}\alpha_{i}e_{i}+\frac{n-3}{4}\alpha_1e_n,\\
D(e_{2})&=&2\sum\limits_{k=2}^{n-2}\alpha_{k-1}e_{k}+(\frac{n-3}{4}\alpha_1+2\alpha_{n-2})e_{n-1}, & \\
D(e_{i})&=&i\sum\limits_{k=i}^{n-1}\alpha_{k-i+1}e_{k}, & 3\leq i\leq n-1,\\
D(e_{n})&=&\frac{n-3}{4}\alpha_1e_{n-2}+\beta_{n-1}e_{n-1}+\frac{n-1}{2}\alpha_1e_{n}.
\end{array}$$
\end{itemize}
\end{prop}
\begin{proof} Let us prove  Theorem for the algebra $\mu_{1,1}^n,$ and for the algebras $\mu_{1,2}^n,\ \mu_{1,3}^n,\ \mu_{1,4}^n$ the proofs are similar.

 Since the algebra $\mu_{1,1}^n$ has two generators $\{e_{1}, e_{n}\}$, any derivation $D$ on $\mu_{1,1}^n$  is completely determined by $D(e_{1}) $ and
$D(e_{n})$.

Let
$$D(e_{1})=\sum_{i=1}^{n}\alpha_{i}e_{i}, \ \mbox{and}\ D(e_{n})=\sum_{j=1}^{n}\beta_{j}e_{j}.$$

Applying the Leibniz rule we have
$$D(e_2)=D(e_1e_1)=D(e_1)e_1+e_1D(e_1)=\sum_{i=1}^{n}\alpha_{i}e_{i}e_1+\sum_{i=1}^{n}\alpha_{i}e_1e_{i}=$$
$$=\sum_{i=2}^{n-1}\alpha_{i-1}e_{i}+\sum_{i=2}^{n-1}\alpha_{i-1}e_{i}=2\sum_{i=2}^{n-1}\alpha_{i-1}e_{i}.$$
$$D(e_3)=D(e_2e_1)=D(e_2)e_1+e_2D(e_1)=2\sum_{i=2}^{n-1}\alpha_{i-1}e_{i}e_1+\sum_{i=1}^{n}\alpha_{i}e_2e_{i}=$$
$$=2\sum_{i=3}^{n-1}\alpha_{i-2}e_{i}+\sum_{i=3}^{n-1}\alpha_{i-2}e_{i}=3\sum_{i=3}^{n-1}\alpha_{i-2}e_{i}.$$

Now by induction we shall show that
$$D(e_i)=i\sum_{k=i}^{n-1}\alpha_{k-i+1}e_{k}, \quad \quad 2\leq i\leq n-1.$$

Suppose the above is true for some $i$. Now we can determine for $i + 1$.
$$D(e_{i+1})=D(e_ie_1)=D(e_i)e_1+e_iD(e_1)=i\sum_{k=i}^{n-1}\alpha_{k-i+1}e_{k}e_1+\sum_{k=1}^{n}\alpha_{k}e_ie_{k}=$$
$$=i\sum_{k=i+1}^{n-1}\alpha_{k-i}e_{k}+\sum_{k=i+1}^{n-1}\alpha_{k-i}e_{k}=(i+1)\sum_{k=i+1}^{n-1}\alpha_{k-i}e_{k}.$$

Further from the properties of derivations we have:
$$0=D(e_1e_n)=D(e_1)e_n+e_1D(e_n)=\sum_{i=1}^{n}\alpha_{i}e_{i}e_n+\sum_{j=1}^{n}\beta_{j}e_1e_{j}=\sum_{j=2}^{n-1}\beta_{j-1}e_{j}.$$
Therefore we have $\beta_j=0,\ 1\leq j\leq n-2$.

If we check the derivations identity for the all elements, we get the correct equality. We complete the proof of theorem.
\end{proof}

All three-dimensional complex nilpotent associative algebras are given by the Theorem \ref{dim2filiform}, and all these algebras are filiform associative algebras. Below we describe derivations of these algebras.

\begin{prop}\label{prop3}
Derivations of 3-dimensional complex filiform associative algebras $A$ are given as follows:
\begin{itemize}

  \item for the algebra $A_1:$
  $$\begin{array}{llll}
D(e_{1})&=&\alpha_{1,1}e_1+\alpha_{2,1}e_2+\alpha_{3,1}e_3\\
D(e_{2})&=&2\alpha_{1,1}e_2, \\
D(e_{3})&=&\alpha_{2,3}e_2+\alpha_{3,3}e_3.\\
\end{array}$$

  \item for the algebra $A_2:$
  $$\begin{array}{llll}
D(e_{1})&=&\alpha_{1,1}e_1+\alpha_{3,1}e_3\\
D(e_{2})&=&\alpha_{2,2}e_2+\alpha_{3,2}e_3, \\
D(e_{3})&=&(\alpha_{1,1}+\alpha_{2,2})e_3.\\
\end{array}$$

\item for the algebra $A_3:$
  $$\begin{array}{llll}
D(e_{1})&=&\alpha_{1,1}e_1+\alpha_{2,1}e_2+\alpha_{3,1}e_3\\
D(e_{2})&=&\alpha_{1,2}e_1+\alpha_{2,2}e_2+\alpha_{3,2}e_3, \\
D(e_{3})&=&(\alpha_{1,1}+\alpha_{2,2})e_3.\\
\end{array}$$

\item for the algebra $A_4^{\alpha}:$
  $$\begin{array}{llll}
D(e_{1})&=&\alpha_{1,1}e_1+\alpha_{2,1}e_2+\alpha_{3,1}e_3\\
D(e_{2})&=&-\alpha \alpha_{2,1}e_1+(\alpha_{1,1}+\alpha_{2,1})e_2+\alpha_{3,2}e_3, \\
D(e_{3})&=&(2\alpha_{1,1}+\alpha_{2,1})e_3.\\
\end{array}$$

\item for the algebra $A_5:$
  $$\begin{array}{llll}
D(e_{1})&=&\alpha_{1,1}e_1+\alpha_{2,1}e_2+\alpha_{3,1}e_3\\
D(e_{2})&=&2\alpha_{1,1}e_2+2\alpha_{2,1}e_3, \\
D(e_{3})&=&3\alpha_{1,1}e_3.\\
\end{array}$$

\end{itemize}
\end{prop}
describe
\begin{proof}  The proof is carrying out by straightforward verification of the Leibniz rule \ref{der}.
\end{proof}

Now we shall describe  derivations of quasi-filiform algebras. It is well known that 3-dimensional quasi-filiform algebras are algebras with zero products, so this case is trivial. In the case of 4-dimensional quasi-filiform algebras we have that $A^3=0$, and this case is also omitted. We shall describe derivations for the case $n\geq 5 $.

\begin{prop}\label{prop5}
Derivations of $5$-dimensional complex naturally graded quasi-filiform  non-split associative algebras are given as follows:

\begin{itemize}
  \item for the algebra $\lambda_{1}:$
$$\begin{array}{llll}
D(e_{1})&=&\alpha_{1,1}e_{1}+\alpha_{2,1}e_{2}+\alpha_{3,1}e_{3}+\alpha_{4,1}e_{4}+\alpha_{5,1}e_{5},\\
D(e_{2})&=&2\alpha_{1,1}e_{2}+2\alpha_{2,1}e_{3}, & \\
D(e_{3})&=&3\alpha_{1,1}e_{3}, & \\
D(e_{4})&=&-\alpha_{4,1}e_{2}+\alpha_{3,4}e_{3}+\frac{3}{2}\alpha_{1,1}e_{4}+\alpha_{5,4}e_{5},\\
D(e_{5})&=&-\alpha_{5,1}e_{2}+\alpha_{3,5}e_{3}-\alpha_{5,4}e_{4}+\frac{3}{2}\alpha_{1,1}e_{5},\\
\end{array}$$

 \item for the algebra $\lambda_{2}:$
$$\begin{array}{llll}
D(e_{1})&=&\alpha_{1,1}e_{1}+\alpha_{2,1}e_{2}+\alpha_{3,1}e_{3}+\alpha_{4,1}e_{4}+\alpha_{5,1}e_{5},\\
D(e_{2})&=&2\alpha_{1,1}e_{2}+(2\alpha_{2,1}+\alpha_{4,1})e_{3}, & \\
D(e_{3})&=&3\alpha_{1,1}e_{3}, & \\
D(e_{4})&=&-\alpha_{5,1}e_{2}+\alpha_{3,4}e_{3}+2\alpha_{1,1}e_{4},\\
D(e_{5})&=&-\alpha_{4,1}e_{2}+\alpha_{3,5}e_{3}+\alpha_{1,1}e_{5},\\
\end{array}$$

\item for the algebra $\lambda_{3}:$
$$\begin{array}{llll}
D(e_{1})&=&\alpha_{1,1}e_{1}+\alpha_{2,1}e_{2}+\alpha_{3,1}e_{3}+\alpha_{4,1}e_{4}+\alpha_{5,1}e_{5},\\
D(e_{2})&=&2\alpha_{1,1}e_{2}+(2\alpha_{2,1}+\alpha_{4,1})e_{3}, & \\
D(e_{3})&=&3\alpha_{1,1}e_{3}, & \\
D(e_{4})&=&\alpha_{3,1}e_{3}+2\alpha_{1,1}e_{4},\\
D(e_{5})&=&-\alpha_{5,1}e_{2}+\alpha_{3,5}e_{3}+\frac32\alpha_{1,1}e_{5},\\
\end{array}$$

\item for the algebra $\lambda_{4}:$
$$\begin{array}{llll}
D(e_{1})&=&\alpha_{2,1}e_{2}+\alpha_{3,1}e_{3}+\alpha_{4,1}e_{4}+\alpha_{5,1}e_{5},\\
D(e_{2})&=&(2\alpha_{2,1}+\alpha_{4,1})e_{3}, & \\
D(e_{3})&=&0, & \\
D(e_{4})&=&-\alpha_{4,1}e_{2}+\alpha_{3,4}e_{3},\\
D(e_{5})&=&-\alpha_{5,1}e_{2}+\alpha_{3,5}e_{3},\\
\end{array}$$

\item for the algebra $\lambda_{5}:$
$$\begin{array}{llll}
D(e_{1})&=&\alpha_{1,1}e_{1}+\alpha_{2,1}e_{2}+\alpha_{3,1}e_{3},\\
D(e_{2})&=&2\alpha_{1,1}e_{2}+2\alpha_{2,1}e_{3}, & \\
D(e_{3})&=&3\alpha_{1,1}e_{3}, & \\
D(e_{4})&=&\alpha_{3,4}e_{3}+\alpha_{4,4}e_{4}+\alpha_{5,4}e_{5},\\
D(e_{5})&=&\alpha_{3,5}e_{3}+\alpha_{4,5}e_{4}+(3\alpha_{1,1}-\alpha_{4,4})e_{5},\\
\end{array}$$

\item for the algebra $\lambda_{6}^{\alpha}:$
$$\begin{array}{llll}
D(e_{1})&=&\alpha_{1,1}e_{1}+\alpha_{2,1}e_{2}+\alpha_{3,1}e_{3},\\
D(e_{2})&=&2\alpha_{1,1}e_{2}+2\alpha_{2,1}e_{3}, & \\
D(e_{3})&=&3\alpha_{1,1}e_{3}, & \\
D(e_{4})&=&\alpha_{3,4}e_{3}+\alpha_{4,4}e_{4}+(3\alpha_{1,1}-2\alpha_{4,4})e_{5},\\
D(e_{5})&=&\alpha_{3,5}e_{3}-\alpha (3\alpha_{1,1}-2\alpha_{4,4})e_{4}+(3\alpha_{1,1}-\alpha_{4,4})e_{5},\\
\end{array}$$

\item for the algebra $\mu_{1}:$
$$\begin{array}{llll}
D(e_{1})&=&\alpha_{1,1}e_1+\alpha_{2,1}e_{2}+\alpha_{3,1}e_{3}+\alpha_{4,1}e_{4}+\alpha_{5,1}e_{5},\\
D(e_{2})&=&2\alpha_{1,1}e_2+2\alpha_{2,1}e_3+\alpha_{4,1}e_{5}, & \\
D(e_{3})&=&3\alpha_{1,1}e_3, & \\
D(e_{4})&=&\alpha_{3,4}e_{3}+\alpha_{4,4}e_{4}+\alpha_{5,4}e_5,\\
D(e_{5})&=&(\alpha_{1,1}+\alpha_{4,4})e_{5},\\
\end{array}$$

\item for the algebra $\mu_{2}:$
$$\begin{array}{llll}
D(e_{1})&=&\alpha_{1,1}e_1+\alpha_{2,1}e_{2}+\alpha_{3,1}e_{3}+\alpha_{4,1}e_{4}+\alpha_{5,5}e_{5},\\
D(e_{2})&=&2\alpha_{1,1}e_2+2\alpha_{2,1}e_3+\alpha_{4,1}e_{5}, & \\
D(e_{3})&=&3\alpha_{1,1}e_3, & \\
D(e_{4})&=&-\alpha_{4,1}e_2+\alpha_{3,4}e_{3}+\frac32\alpha_{1,1}e_{4}+\alpha_{5,4}e_5,\\
D(e_{5})&=&\frac52\alpha_{1,1}e_{5},\\
\end{array}$$

\item for the algebra $\mu_{3}:$
$$\begin{array}{llll}
D(e_{1})&=&\alpha_{1,1}e_1+\alpha_{2,1}e_{2}+\alpha_{3,1}e_{3}+\alpha_{4,1}e_{4}+\alpha_{5,1}e_{5},\\
D(e_{2})&=&2\alpha_{1,1}e_2+(2\alpha_{2,1}+\alpha_{5,1})e_3+\alpha_{4,1}e_{5}, & \\
D(e_{3})&=&(3\alpha_{1,1}+\alpha_{4,1})e_3, & \\
D(e_{4})&=&\alpha_{3,4}e_{3}+(\alpha_{1,1}+\alpha_{4,1})e_{4}+\alpha_{5,4}e_5,\\
D(e_{5})&=&(\alpha_{2,1}+\alpha_{5,4})e_{3}+(2\alpha_{1,1}+\alpha_{4,1})e_{4},\\
\end{array}$$

\item for the algebra $\mu_{4}:$
$$\begin{array}{llll}
D(e_{1})&=&\alpha_{2,1}e_{2}+\alpha_{3,1}e_{3}+\alpha_{4,1}e_{4}+\alpha_{5,1}e_{5},\\
D(e_{2})&=&(2\alpha_{2,1}+\alpha_{5,1})e_3+\alpha_{4,1}e_{5}, & \\
D(e_{3})&=&\alpha_{4,1}e_3, & \\
D(e_{4})&=&-\alpha_{4,1}e_2+\alpha_{3,4}e_{3}+\alpha_{4,1}e_{4}+\alpha_{5,4}e_5,\\
D(e_{5})&=&(\alpha_{2,1}+\alpha_{5,4})e_{3}+\alpha_{4,1}e_{4},\\
\end{array}$$

\item for the algebra $\mu_{5}:$
$$\begin{array}{llll}
D(e_{1})&=&\alpha_{1,1}e_{1}+\alpha_{2,1}e_{2}+\alpha_{3,1}e_{3}+\alpha_{4,1}e_{4}+\alpha_{5,1}e_{5},\\
D(e_{2})&=&2\alpha_{1,1}e_{2}+(2\alpha_{2,1}+\alpha_{4,1})e_3+2\alpha_{4,1}e_{5}, & \\
D(e_{3})&=&3\alpha_{1,1}e_3, & \\
D(e_{4})&=&\alpha_{2,4}e_2+\alpha_{3,4}e_{3}+2\alpha_{1,1}e_{4}+\alpha_{5,4}e_5,\\
D(e_{5})&=&\alpha_{2,4}e_{3}+3\alpha_{1,1}e_{5},\\
\end{array}$$

\item for the algebra $\mu_{6}:$
$$\begin{array}{llll}
D(e_{1})&=&\alpha_{2,1}e_{2}+\alpha_{3,1}e_{3}+\alpha_{4,1}e_{4}+\alpha_{5,1}e_{5},\\
D(e_{2})&=&(2\alpha_{2,1}+\alpha_{4,1})e_3+\alpha_{4,1}e_{5}, & \\
D(e_{3})&=&0, & \\
D(e_{4})&=&\alpha_{2,4}e_2+\alpha_{3,4}e_{3}+\alpha_{5,4}e_5,\\
D(e_{5})&=&(\alpha_{4,1}+\alpha_{2,4})e_{3},\\
\end{array}$$

\item for the algebra $\mu_{7}^{\alpha\neq 1}:$
$$\begin{array}{llll}
D(e_{1})&=&\alpha_{1,1}e_{1}+\alpha_{2,1}e_{2}+\alpha_{3,1}e_{3}+\alpha_{4,1}e_{4}+\alpha_{5,1}e_{5},\\
D(e_{2})&=&2\alpha_{1,1}e_{2}+2\alpha_{2,1}e_3+(1+\alpha)\alpha_{4,1}e_{5}, & \\
D(e_{3})&=&3\alpha_{1,1}e_3, & \\
D(e_{4})&=&\alpha_{3,4}e_{3}+\alpha_{4,4}e_{4}+\alpha_{5,4}e_5,\\
D(e_{5})&=&(\alpha_{1,1}+\alpha_{4,4})e_{5},\\
\end{array}$$

\item for the algebra $\mu_{7}^{1}:$
$$\begin{array}{llll}
D(e_{1})&=&\alpha_{1,1}e_{1}+\alpha_{2,1}e_{2}+\alpha_{3,1}e_{3}+\alpha_{4,1}e_{4}+\alpha_{5,1}e_{5},\\
D(e_{2})&=&2\alpha_{1,1}e_{2}+2\alpha_{2,1}e_3+2\alpha_{4,1}e_{5}, & \\
D(e_{3})&=&3\alpha_{1,1}e_3, & \\
D(e_{4})&=&\alpha_{2,4}e_{2}+\alpha_{3,4}e_{3}+\alpha_{4,4}e_{4}+\alpha_{5,4}e_5,\\
D(e_{5})&=&\alpha_{2,4}e_{3}+(\alpha_{1,1}+\alpha_{4,4})e_{5},\\
\end{array}$$

\item for the algebra $\mu_{8}^{\alpha\neq 1}:$
$$\begin{array}{llll}
D(e_{1})&=&\alpha_{1,1}e_{1}+\alpha_{2,1}e_{2}+\alpha_{3,1}e_{3}+\alpha_{4,1}e_{4}+\alpha_{5,1}e_{5},\\
D(e_{2})&=&2\alpha_{1,1}e_{2}+2\alpha_{2,1}e_3+(1+\alpha)\alpha_{4,1}e_{5}, & \\
D(e_{3})&=&3\alpha_{1,1}e_3, & \\
D(e_{4})&=&-\alpha_{4,1}e_{2}+\alpha_{3,4}e_{3}+\frac32\alpha_{1,1}e_{4}+\alpha_{5,4}e_5,\\
D(e_{5})&=&\frac52\alpha_{1,1}e_{5},\\
\end{array}$$

\item for the algebra $\mu_{8}^{1}:$
$$\begin{array}{llll}
D(e_{1})&=&\alpha_{1,1}e_{1}+\alpha_{2,1}e_{2}+\alpha_{3,1}e_{3}+\alpha_{4,1}e_{4}+\alpha_{5,1}e_{5},\\
D(e_{2})&=&2\alpha_{1,1}e_{2}+2\alpha_{2,1}e_3+2\alpha_{4,1}e_{5}, & \\
D(e_{3})&=&3\alpha_{1,1}e_3, & \\
D(e_{4})&=&\alpha_{2,4}e_{2}+\alpha_{3,4}e_{3}+\frac32\alpha_{1,1}e_{4}+\alpha_{5,4}e_5,\\
D(e_{5})&=&(\alpha_{4,1}+\alpha_{2,4})e_{3}+\frac52\alpha_{1,1}e_{5},\\
\end{array}$$

\item for the algebra $\mu_{9}:$
$$\begin{array}{llll}
D(e_{1})&=&\alpha_{1,1}e_{1}+\alpha_{2,1}e_{2}+\alpha_{3,1}e_{3}+\alpha_{5,1}e_{5},\\
D(e_{2})&=&2\alpha_{1,1}e_{2}+2\alpha_{2,1}e_3, & \\
D(e_{3})&=&3\alpha_{1,1}e_3, & \\
D(e_{4})&=&\alpha_{3,4}e_{3}+2\alpha_{1,1}e_{4}+\alpha_{5,4}e_5,\\
D(e_{5})&=&4\alpha_{1,1}e_{5},\\
\end{array}$$

\item for the algebra $\mu_{10}:$
$$\begin{array}{llll}
D(e_{1})&=&\alpha_{1,1}e_{1}+\alpha_{2,1}e_{2}+\alpha_{3,1}e_{3}+\alpha_{5,1}e_{5},\\
D(e_{2})&=&2\alpha_{1,1}e_{2}+2\alpha_{2,1}e_3, & \\
D(e_{3})&=&3\alpha_{1,1}e_3, & \\
D(e_{4})&=&\alpha_{3,4}e_{3}+\alpha_{1,1}e_{4}+\alpha_{5,4}e_5,\\
D(e_{5})&=&2\alpha_{1,1}e_{5},\\
\end{array}$$

\item for the algebra $\mu_{11}:$
$$\begin{array}{llll}
D(e_{1})&=&\alpha_{2,1}e_{2}+\alpha_{3,1}e_{3}+\alpha_{5,1}e_{5},\\
D(e_{2})&=&2\alpha_{2,1}e_3, & \\
D(e_{3})&=&0, & \\
D(e_{4})&=&\alpha_{3,4}e_{3}+\alpha_{5,4}e_5,\\
D(e_{5})&=&0,\\
\end{array}$$

\item for the algebra $\mu_{12}:$
$$\begin{array}{llll}
D(e_{1})&=&\alpha_{1,1}e_{1}+\alpha_{2,1}e_{2}+\alpha_{3,1}e_{3}+\alpha_{4,1}e_{4}+\alpha_{5,1}e_{5},\\
D(e_{2})&=&(2\alpha_{1,1}+\alpha_{4,1})e_{2}+(2\alpha_{2,1}+\alpha_{5,1})e_3, & \\
D(e_{3})&=&(3\alpha_{1,1}+\alpha_{4,1})e_3, & \\
D(e_{4})&=&\alpha_{2,4}e_{2}+\alpha_{3,4}e_{3}+(\alpha_{1,1}+\alpha_{4,1})e_{4}+(2\alpha_{2,1}+\alpha_{5,1}-2\alpha_{2,4})e_5,\\
D(e_{5})&=&\alpha_{2,4}e_{3}+(2\alpha_{1,1}+\alpha_{4,1})e_{5},\\
\end{array}$$

\item for the algebra $\mu_{13}:$
$$\begin{array}{llll}
D(e_{1})&=&\alpha_{1,1}e_{1}+\alpha_{2,1}e_{2}+\alpha_{3,1}e_{3}+\alpha_{4,1}e_{4}+\alpha_{5,1}e_{5},\\
D(e_{2})&=&3\alpha_{1,1}e_{2}+(2\alpha_{2,1}+\alpha_{5,1})e_3, & \\
D(e_{3})&=&4\alpha_{1,1}e_3, & \\
D(e_{4})&=&\alpha_{2,4}e_{2}+\alpha_{3,4}e_{3}+2\alpha_{1,1}e_{4}+(-2\alpha_{1,1}+2\alpha_{2,1}+\alpha_{5,1}-2\alpha_{2,4})e_5,\\
D(e_{5})&=&(\alpha_{1,1}+\alpha_{2,4})e_{3}+3\alpha_{1,1}e_{5},\\
\end{array}$$

\item for the algebra $\mu_{14}:$
$$\begin{array}{llll}
D(e_{1})&=&\alpha_{1,1}e_{1}+\alpha_{2,1}e_{2}+\alpha_{3,1}e_{3}+\alpha_{5,1}e_{5},\\
D(e_{2})&=&2\alpha_{1,1}e_{2}+(2\alpha_{2,1}+\alpha_{5,1})e_3, & \\
D(e_{3})&=&3\alpha_{1,1}e_3, & \\
D(e_{4})&=&\alpha_{2,4}e_{2}+\alpha_{3,4}e_{3}+\alpha_{1,1}e_{4}+\alpha_{5,1}e_5,\\
D(e_{5})&=&\alpha_{2,4}e_{3}+2\alpha_{1,1}e_{5},\\
\end{array}$$

\item for the algebra $\mu_{15}:$
$$\begin{array}{llll}
D(e_{1})&=&\alpha_{2,1}e_{2}+\alpha_{3,1}e_{3}+\alpha_{5,1}e_{5},\\
D(e_{2})&=&(2\alpha_{2,1}+\alpha_{5,1})e_3, & \\
D(e_{3})&=&0, & \\
D(e_{4})&=&\alpha_{2,4}e_{2}+\alpha_{3,4}e_{3}+\alpha_{5,1}e_5,\\
D(e_{5})&=&\alpha_{2,4}e_{3},\\
\end{array}$$

\item for the algebra $\mu_{16}:$
$$\begin{array}{llll}
D(e_{1})&=&\alpha_{1,1}e_{1}+\alpha_{2,1}e_{2}+\alpha_{3,1}e_{3}+\alpha_{5,1}e_{5},\\
D(e_{2})&=&2\alpha_{1,1}e_{2}+2\alpha_{2,1}e_3, & \\
D(e_{3})&=&3\alpha_{1,1}e_3, & \\
D(e_{4})&=&\alpha_{2,4}e_{2}+\alpha_{3,4}e_{3}+\alpha_{1,1}e_{4}+\alpha_{2,1}e_5,\\
D(e_{5})&=&(\alpha_{5,1}+\alpha_{2,4})e_{3}+2\alpha_{1,1}e_{5},\\
\end{array}$$

\item for the algebra $\mu_{17}:$
$$\begin{array}{llll}
D(e_{1})&=&\alpha_{2,1}e_{2}+\alpha_{3,1}e_{3}+\alpha_{5,1}e_{5},\\
D(e_{2})&=&2\alpha_{2,1}e_3, & \\
D(e_{3})&=&0, & \\
D(e_{4})&=&\alpha_{2,4}e_{2}+\alpha_{3,4}e_{3}+\alpha_{2,1}e_5,\\
D(e_{5})&=&(\alpha_{5,1}+\alpha_{2,4})e_{3},\\
\end{array}$$

\item for the algebra $\mu_{18}:$
$$\begin{array}{llll}
D(e_{1})&=&\alpha_{1,1}e_{1}+\alpha_{3,1}e_{3}+\alpha_{5,1}e_{5},\\
D(e_{2})&=&2\alpha_{1,1}e_{2}, & \\
D(e_{3})&=&3\alpha_{1,1}e_3, & \\
D(e_{4})&=&\alpha_{3,4}e_{3}+\alpha_{1,1}e_{4}+\alpha_{5,4}e_5,\\
D(e_{5})&=&-\alpha_{5,1}e_{3}+2\alpha_{1,1}e_{5},\\
\end{array}$$

\item for the algebra $\mu_{19}:$
$$\begin{array}{llll}
D(e_{1})&=&\alpha_{1,1}e_{1}+\alpha_{2,1}e_{2}+\alpha_{3,1}e_{3}-\alpha_{1,1}e_{4}+\alpha_{5,1}e_{5},\\
D(e_{2})&=&\alpha_{1,1}e_{2}+(2\alpha_{2,1}+\alpha_{5,1})e_3, & \\
D(e_{3})&=&\alpha_{1,1}e_3, & \\
D(e_{4})&=&\alpha_{2,4}e_{2}+\alpha_{3,4}e_{3}+(\alpha_{1,1}+\alpha_{2,1}+\alpha_{5,1}-\alpha_{2,4})e_5,\\
D(e_{5})&=&(-\alpha_{1,1}+\alpha_{2,4})e_{3},\\
\end{array}$$

\item for the algebra $\mu_{20}:$
$$\begin{array}{llll}
D(e_{1})&=&\alpha_{1,1}e_{1}+\alpha_{2,1}e_{2}+\alpha_{3,1}e_{3}+\alpha_{5,1}e_{5},\\
D(e_{2})&=&2\alpha_{1,1}e_{2}+(2\alpha_{2,1}+\alpha_{5,1})e_3, & \\
D(e_{3})&=&3\alpha_{1,1}e_3, & \\
D(e_{4})&=&(-2\alpha_{1,1}+4\alpha_{2,1}+3\alpha_{5,1})e_{2}+\alpha_{3,1}e_{3}+\alpha_{1,1}e_{4}+(\alpha_{1,1}-\alpha_{2,1})e_5,\\
D(e_{5})&=&(-2\alpha_{1,1}+4\alpha_{2,1}+2\alpha_{5,1})e_{3}+2\alpha_{1,1}e_{5},\\
\end{array}$$

\item for the algebra $\mu_{21}^{\alpha},\ \alpha\in \{\pm\imath\}:$
$$\begin{array}{llll}
D(e_{1})&=&\frac12((\alpha-1)\alpha_{5,1}+\alpha_{3,2})e_{2}+\alpha_{3,1}e_{3}+\alpha_{5,1}e_{5},\\
D(e_{2})&=&\alpha_{3,2}e_3, & \\
D(e_{3})&=&0, & \\
D(e_{4})&=&\frac{(\alpha-1)\alpha_{5,1}+\alpha\alpha_{3,2}+(1-\alpha)\alpha_{5,4}}{\alpha-1}e_2+\alpha_{3,4}e_{3}+\alpha_{5,4}e_5,\\
D(e_{5})&=&\frac{\alpha(2\alpha_{5,1}+\alpha_{3,2})+(1-\alpha)\alpha_{5,4}}{\alpha-1}e_3,\\
\end{array}$$

\item for the algebra $\mu_{22}^{1},:$
$$\begin{array}{llll}
D(e_{1})&=&\alpha_{1,1}e_{1}+\alpha_{2,1}e_{2}+\alpha_{3,1}e_{3}+\alpha_{4,1}e_{4}+\alpha_{5,1}e_{5},\\
D(e_{2})&=&2\alpha_{1,1}e_{2}+2\alpha_{2,1}e_3+2\alpha_{4,1}e_{5}, & \\
D(e_{3})&=&3\alpha_{1,1}e_3, & \\
D(e_{4})&=&2\alpha_{4,1}e_1+\alpha_{2,4}e_2+\alpha_{3,4}e_{3}+(\alpha_{1,1}-\alpha_{4,1})e_{4}+(\alpha_{2,1}+\alpha_{5,1}-\alpha_{2,4})e_5,\\
D(e_{5})&=&\alpha_{4,1}e_2+(-\alpha_{5,1}+\alpha_{2,4})e_{3}+(2\alpha_{1,1}+\alpha_{4,1})e_{5},\\
\end{array}$$

\item for the algebra $\mu_{22}^{0},:$
$$\begin{array}{llll}
D(e_{1})&=&\alpha_{1,1}e_{1}+\alpha_{2,1}e_{2}+\alpha_{3,1}e_{3}+\alpha_{4,1}e_{4}+\alpha_{5,1}e_{5},\\
D(e_{2})&=&(2\alpha_{1,1}+\alpha_{4,1})e_{2}+(2\alpha_{2,1}+\alpha_{5,1})e_{3}+\alpha_{4,1}e_{5}, & \\
D(e_{3})&=&(3\alpha_{1,1}+2\alpha_{4,1})e_3, & \\
D(e_{4})&=&\alpha_{2,4}e_2+\alpha_{3,4}e_{3}+(\alpha_{1,1}+\alpha_{4,1})e_{4}+(\alpha_{2,1}+\alpha_{5,1}-\alpha_{2,4})e_5,\\
D(e_{5})&=&\alpha_{2,4}e_{3}+2(\alpha_{1,1}+\alpha_{4,1})e_{5},\\
\end{array}$$

\item for the algebra $\mu_{22}^{\alpha},\ \alpha\neq0;1:$
$$\begin{array}{llll}
D(e_{1})&=&((\alpha^2+\alpha-1)\alpha_{4,1}+\alpha_{4,4})e_{1}-\alpha_{5,1}e_{2}+\alpha_{3,1}e_{3}+\alpha_{4,1}e_{4}+\alpha_{5,1}e_{5},\\
D(e_{2})&=&((2\alpha^2+\alpha-1)\alpha_{1,1}+2\alpha_{4,4})e_{2}-(1+\alpha)e_{3}+(1+\alpha)\alpha_{4,1}e_{5}, & \\
D(e_{3})&=&((2\alpha^2+2\alpha-1)\alpha_{1,1}+3\alpha_{4,4})e_3, & \\
D(e_{4})&=&(\alpha^2+\alpha)\alpha_{4,1}e_1-\alpha_{5,4}e_2+\alpha_{3,4}e_{3}+\alpha_{4,4}e_{4}+\alpha_{5,4}e_5,\\
D(e_{5})&=&\alpha^2\alpha_{4,1}e_2-(\alpha\alpha_{5,1}+\alpha_{5,4})e_{3}+((\alpha^2+2\alpha)\alpha_{4,1}+2\alpha_{4,4})e_{5}.\\
\end{array}$$
\end{itemize}
\end{prop}

\begin{proof} The proof follows by straightforward calculations similarly to
the proof of Proposition \ref{prop}.
\end{proof}

\begin{prop}\label{prop2}
The derivations of the $n$-dimensional $(n \geq 6)$ complex naturally graded quasi-filiform  non-split associative algebras are given as follows:

\begin{itemize}
  \item for the algebra $\mu_{2,1}^n$
$$\begin{array}{llll}
D(e_{1})&=&\sum\limits_{i=1}^{n}\alpha_{i}e_{i},\\
D(e_{2})&=&2\sum\limits_{k=2}^{n-2}\alpha_{k-1}e_{k}+\alpha_{n-1}e_{n}, & \\
D(e_{i})&=&i\sum\limits_{k=i}^{n-2}\alpha_{k-i+1}e_{k}, & 3\leq i\leq n-2,\\
D(e_{n-1})&=&\beta_{n-2}e_{n-2}+\beta_{n-1}e_{n-1}+\beta_ne_n,\\
D(e_{n})&=&(\alpha_1+\beta_{n-1})e_{n}.
\end{array}$$
\item for the algebra $\mu_{2,2}^n(\alpha=1)$
    $$\begin{array}{llll}
    D(e_{1})&=&\sum\limits_{i=1}^{n}\alpha_{i}e_{i},\\
    D(e_{2})&=&2\sum\limits_{k=2}^{n-2}\alpha_{k-1}e_{k}+2\alpha_{n-1}e_{n}, & \\
    D(e_{i})&=&i\sum\limits_{k=i}^{n-2}\alpha_{k-i+1}e_{k}, & 3\leq i\leq n-2,\\
    D(e_{n-1})&=&\sum\limits_{j=1}^{n-2}\beta_{j}e_{j}+\beta_{n}e_{n},\\
    D(e_{n})&=&\sum\limits_{j=2}^{n-2}\beta_{j-1}e_{j}+\alpha_1e_n.
    \end{array}$$

\item  for the algebra $\mu_{2,2}^n(\alpha=-1)$
    $$\begin{array}{llll}
    D(e_{1})&=&\sum\limits_{i=1}^{n}\alpha_{i}e_{i},\\
    D(e_{i})&=&i\sum\limits_{k=i}^{n-2}\alpha_{k-i+1}e_{k}, & 2\leq i\leq n-2,\\
    D(e_{n-1})&=&\beta_{n-2}e_{n-2}+\beta_{n-1}e_{n-1}+\beta_ne_n,\\
    D(e_{n})&=&(\alpha_1+\beta_{n-1})e_{n}.
    \end{array}$$

\item  for the algebra $\mu_{2,2}^n(\alpha\neq -1;1)$
    $$\begin{array}{llll}
    D(e_{1})&=&\sum\limits_{i=1}^{n}\alpha_{i}e_{i},\\
    D(e_{2})&=&2\sum\limits_{k=2}^{n-2}\alpha_{k-1}e_{k}+(1+\alpha)\alpha_{n-1}e_{n}, & \\
    D(e_{i})&=&i\sum\limits_{k=i}^{n-2}\alpha_{k-i+1}e_{k}, & 3\leq i\leq n-2,\\
    D(e_{n-1})&=&\beta_{n-2}e_{n-2}+\beta_ne_n,\\
    D(e_{n})&=&\alpha_1e_{n}.
    \end{array}$$

\item for the algebra $\mu_{2,3}^n$
$$\begin{array}{llll}
    D(e_{1})&=&\sum\limits_{i=1}^{n-2}\alpha_{i}e_{i}+\alpha_ne_n,\\
    D(e_{i})&=&i\sum\limits_{k=i}^{n-2}\alpha_{k-i+1}e_{k}, & 2\leq i\leq n-2,\\
    D(e_{n-1})&=&\beta_{n-2}e_{n-2}+\beta_{n-1}e_{n-1}+\beta_ne_n,\\
    D(e_{n})&=&2\beta_{n-1}e_n.
\end{array}$$

\item for the algebra $\mu_{2,4}^n$
$$\begin{array}{llll}
    D(e_{1})&=&\sum\limits_{i=1}^{n-2}\alpha_{i}e_{i}+\alpha_ne_n,\\
    D(e_{i})&=&i\sum\limits_{k=i}^{n-2}\alpha_{k-i+1}e_{k}, & 2\leq i\leq n-2,\\
    D(e_{n-1})&=&\beta_{n-2}e_{n-2}+\alpha_{1}e_{n-1}+\beta_ne_n,\\
    D(e_{n})&=&2\alpha_{1}e_n.
\end{array}$$
\end{itemize}
\end{prop}

\begin{proof} The proof follows by straightforward calculations similarly to
the proof of Proposition \ref{prop}.
\end{proof}

\section{Local and 2-local derivation on  nilpotent associative
algebras of small dimension}

Now we study local and 2-local derivations on nilpotent associative
algebras of small dimension.

\begin{thm}\label{small3}
Local derivations of 3-dimensional complex nilpotent associative algebras $\mathcal{A}$ are given as follows:
\begin{itemize}
  \item for the algebra $\mathcal{A}_1$
  $$\begin{array}{llll}
\Delta(e_{1})&=&c_{1,1}e_1,+c_{2,1}e_2+c_{3,1}e_3,\\
\Delta(e_{2})&=&c_{2,2}e_2, \\
\Delta(e_{3})&=&c_{2,3}e_2+c_{3,3}e_3.\\
\end{array}$$
  \item for the algebra $\mathcal{A}_2:$
  $$\begin{array}{llll}
\Delta(e_{1})&=&c_{1,1}e_1+c_{3,1}e_3,\\
\Delta(e_{2})&=&c_{2,2}e_2+c_{3,2}ce_3, \\
D(e_{3})&=&c_{3,3}e_3.\\
\end{array}$$

\item for the algebra $\mathcal{A}_3:$
  $$\begin{array}{llll}
\Delta(e_{1})&=&c_{1,1}e_1+c_{2,1}e_2+c_{3,1}e_3,\\
\Delta(e_{2})&=&c_{1,2}e_1+c_{2,2}e_2+c_{3,2}e_3, \\
\Delta(e_{3})&=&c_{3,3}e_3.\\
\end{array}$$

\item for the algebra $\mathcal{A}_4^{\alpha}:$
  $$\begin{array}{llll}
\Delta(e_{1})&=&c_{1,1}e_1+c_{2,1}e_2+c_{3,1}e_3,\\
\Delta(e_{2})&=&\alpha c_{2,1}e_1+c_{2,2}e_2+c_{3,2}e_3, \\
\Delta(e_{3})&=&c_{3,3}e_3.\\
\end{array}$$

\item for the algebra $\mathcal{A}_5:$
  $$\begin{array}{llll}
\Delta(e_{1})&=&c_{1,1}e_1+c_{2,1}e_2+c_{3,1}e_3,\\
\Delta(e_{2})&=&c_{2,2}e_2+c_{3,2}e_3, \\
\Delta(e_{3})&=&c_{3,3}e_3.\\
\end{array}$$
\end{itemize}
\end{thm}

\begin{proof}
We proof the Theorem for the algebra $\mathcal{A}_1,$ and for the algebras $\mathcal{A}_2,\ \mathcal{A}_3,\ \mathcal{A}_4^{\alpha},\ \mathcal{A}_5$ the proofs are similar.

  Given an arbitrary local derivation $\Delta$ on $\mathcal{A}_1$, let $\mathfrak{C}$ be the matrix of $\Delta$:

\[
\mathfrak{C}=\left(
              \begin{array}{ccc}
                c_{1,1}   & c_{1,2}        &  c_{1,3}       \\
                c_{2,1}   & c_{2,2}        &  c_{2,3}       \\
                c_{3,1}   & c_{3,2}        &  c_{3,3}       \\
              \end{array}
            \right)
\]

By the definition for all $x=\sum\limits_{i=1}^3x_ie_i\in\mathcal{A}_1$
there exists a derivation $D_x$ on $\mathcal{A}_1$ such that
$$
\Delta(x)=D_x(x).
$$
By Proposition \ref{prop3}, the derivation $D_x$ has the following matrix form:

\[
\mathfrak{C}_x=\left(
              \begin{array}{ccc}
 \alpha_{1,1}^x    & 0               & 0                                 \\
 \alpha_{2,1}^x    & 2\alpha_{1,1}^x & \alpha_{2,3}^x                                 \\
 \alpha_{3,1}^x    & 0               & \alpha_{3,3}^x                      \\
  \end{array}
            \right)
\]

 For the matrix $\mathfrak{C}$   of $\Delta$  by choosing subsequently $x=e_1,\ x=e_2,\ x=e_3$ from  $\Delta(x)=D_x(x),$ we have $\mathfrak{C}\overline{x}=D_x(\overline{x}),$
where $\overline{x}$ is the vector corresponding to $x$. This implies

\[
\mathfrak{C}=\left(
              \begin{array}{ccc}
                c_{1,1}   & 0              & 0     \\
                c_{2,1}   & c_{2,2}        & c_{2,3}     \\
                c_{3,1}   & 0              & c_{3,3}       \\
              \end{array}
            \right)
\]

Using again $\Delta(x)=D_x(x),$ i.e. $\mathfrak{C}\overline{x}=\mathfrak{C}_x(\overline{x}),$ where $\overline{x}$ is the vector corresponding to $x=\sum\limits_{i=1}^3x_ie_i,$ we obtain the system of equalities
\begin{equation}\label{3sm}
  \begin{array}{ll}
    c_{1,1}x_1 &=\alpha_{1,1}^xx_1, \\
    c_{2,1}x_1+c_{2,2}x_2+c_{2,3}x_3 &= \alpha_{2,1}^xx_1+2\alpha_{1,1}^xx_2+\alpha_{2,3}^xx_3,\\
    c_{3,1}x_1+c_{3,3}x_3&=\alpha_{3,1}^xx_1+\alpha_{3,3}^xx_{3},\\
    \end{array}
\end{equation}

Let us consider the following three cases:

\textbf{Case 1:} Let $x_1\neq 0,$ then putting $\alpha_{2,3}^x=\alpha_{3,3}^x=0$ from (\ref{3sm}) we uniquely determine
$$\alpha_{1,1}^x=c_{1,1},\ \ \alpha_{2,1}^x=\frac{c_{2,1}x_1+c_{2,2}x_2+c_{2,3}x_3-2c_{1,1}x_2}{x_1},\ \ \alpha_{3,1}^x=\frac{c_{3,1}x_1+c_{3,3}x_3}{x_1}.$$

\textbf{Case 2:} Let $x_1=0$ and $x_2\neq0,$ then putting $\alpha_{2,3}^x=\alpha_{3,3}^x=0$ from (\ref{3sm}) we uniquely determine
$$\alpha_{1,1}^x=\frac{c_{2,2}x_2+c_{2,3}x_3}{2x_2}.$$

\textbf{Case 3:} Let $x_1=x_2=0$ and $x_3\neq0,$ then
$\alpha_{2,3}^x=c_{2,3},\ \alpha_{3,3}^x=c_{3,3}.$

We find $\alpha_{1,1}^x,\ \alpha_{2,1}^x,\ \alpha_{2,3}^x,\ \alpha_{3,1}^x$ and $\alpha_{3,3}^x$. So,
the needed derivation $D_x$ is defined. The proof is complete.
\end{proof}

By direct  calculation we obtain the dimensions of the spaces of derivation and
local derivations to 3-dimensional complex nilpotent associative algebras $\mathcal{A}$.

$$\begin{tabular}{|p{2.5cm}|p{4cm}|p{4cm}|}
    \hline
  Algebra &  The dimensions of the space of derivations & The dimensions of the space of local derivations \\
  \hline
    $\mathcal{A}_1$  & $5$& $6$ \\[1mm]
  \hline
    $\mathcal{A}_2$  & $4$& $5$ \\[1mm]
  \hline
    $\mathcal{A}_3$   & $6$& $7$ \\[1mm]
  \hline
    $\mathcal{A}_4^{\alpha}$   & $4$& $7$ \\[1mm]
  \hline
    $\mathcal{A}_5$   & $3$& $6$ \\[1mm]
  \hline
    \end{tabular}$$

\begin{cor} The null-filiform associative algebras admit local derivations which
are not derivations.
\end{cor}

Now we investigate 2-local derivations on $3$-dimensional complex filiform associative algebras.

\begin{thm}\label{small31}
$3$-dimensional complex filiform associative algebras admit $2$-local derivations which are not derivations.
\end{thm}

\begin{proof}  We shall prove this theorem for the algebra $\mathcal{A}_1;$  for the algebras $\mathcal{A}_2,\ \mathcal{A}_3,\ \mathcal{A}_4^{\alpha},\ \mathcal{A}_5$ the proofs are similar.

Let us define a homogeneous non additive function $f$ on $\mathbb{C}^{2}$ as follows

\[ f(z_{1},z_{3}) = \begin{cases}
\frac{z^{2}_{1}}{z_{3}}, & \text{if $z_{n}\neq0$,}\\
0, & \text{if $z_{n}=0$}.
\end{cases} \]
where $(z_1,z_3)\in\mathbb{C}.$
Consider the map $\nabla: \mathcal{A}_1\rightarrow \mathcal{A}_1$ defined by the rule
\begin{equation*}\label{for1.1.1}
\nabla(x)=f(x_{1},x_{3})e_{3}, \quad
\textmd{where} \quad    x=x_{1}e_{1}+x_2e_2+x_3e_3\in \mathcal{A}_1.
\end{equation*}
Since $f$ is not additive, $\nabla$ is not a derivation.

Let us show that $\nabla$ is a $2$-local derivation.
For the elements
$$x=x_{1}e_{1}+x_2e_2+x_3e_3, \quad
y=y_{1}e_{1}+y_2e_2+y_3e_3,
$$
we search a derivation $D$ in the form:
\begin{equation*}
D(e_1)=\alpha_{3,1}e_3, \quad  D(e_2)=0, \quad D(e_3)=\alpha_{3,3} e_3.
\end{equation*}

Assume that $\nabla(x)=D(x)$ and $\nabla(y)=D(y)$. Then
we obtain the following system of equations for $\alpha_n$ and
$\beta:$
\begin{equation}\label{for1.1.2}
\begin{cases}
x_1 \alpha_{3,1}+x_3\alpha_{3,3} =f(x_1, x_3), \\
y_1\alpha_{3,1}+y_3\alpha_{3,3} = f(y_1, y_3).
\end{cases}
\end{equation}

\textbf{Case 1}. Let $x_{1}y_{3}-x_{3}y_{1}=0$, then the system has infinitely many solutions, because of the right-hand side of system is homogeneous.

\textbf{Case 2}. Let $x_{1}y_{3}-x_{3}y_{1}\neq0$, then the system has a unique solution.
The proof is complete.
\end{proof}

Now we consider local derivations of quasi-filiform algebras.

\begin{thm}\label{thm12}
Local derivations of the $5$-dimensional complex naturally graded quasi-filiform  non-split associative algebras are given as follows:

\begin{itemize}
  \item for the algebra $\lambda_{1}:$
$$\begin{array}{llll}
\Delta(e_{1})&=&c_{1,1}e_{1}+c_{2,1}e_{2}+c_{3,1}e_{3}+c_{4,1}e_{4}+c_{5,1}e_{5},\\
\Delta(e_{2})&=&c_{2,2}e_{2}+c_{3,2}e_{3}, & \\
\Delta(e_{3})&=&c_{3,3}e_{3}, & \\
\Delta(e_{4})&=&c_{2,4}e_{2}+c_{3,4}e_{3}+c_{4,4}e_{4}+c_{5,4}e_{5},\\
\Delta(e_{5})&=&c_{2,5}e_{2}+c_{3,5}e_{3}+c_{4,5}e_{4}+c_{5,5}e_{5},\\
\end{array}$$

 \item for the algebra $\lambda_{2}:$
$$\begin{array}{llll}
\Delta(e_{1})&=&c_{1,1}e_{1}+c_{2,1}e_{2}+c_{3,1}e_{3}+c_{4,1}e_{4}+c_{5,1}e_{5},\\
\Delta(e_{2})&=&c_{2,2}e_{2}+c_{3,2}e_{3}, & \\
\Delta(e_{3})&=&c_{3,3}e_{3}, & \\
\Delta(e_{4})&=&c_{2,4}e_{2}+c_{3,4}e_{3}+c_{4,4}e_{4},\\
\Delta(e_{5})&=&c_{2,5}e_{2}+c_{3,5}e_{3}+c_{5,5}e_{5},\\
\end{array}$$

\item for the algebra $\lambda_{3}:$
$$\begin{array}{llll}
\Delta(e_{1})&=&c_{1,1}e_{1}+c_{2,1}e_{2}+c_{3,1}e_{3}+c_{4,1}e_{4}+c_{5,1}e_{5},\\
\Delta(e_{2})&=&c_{2,2}e_{2}+c_{3,2}e_{3}, & \\
\Delta(e_{3})&=&c_{3,3}e_{3}, & \\
\Delta(e_{4})&=&c_{3,4}e_{3}+c_{4,4}e_{4},\\
\Delta(e_{5})&=&c_{2,5}e_{2}+c_{3,5}e_{3}+c_{5,5}e_{5},\\
\end{array}$$

\item for the algebra $\lambda_{4}:$
$$\begin{array}{llll}
\Delta(e_{1})&=&c_{2,1}e_{2}+c_{3,1}e_{3}+c_{4,1}e_{4}+c_{5,1}e_{5},\\
\Delta(e_{2})&=&c_{3,3}e_{3}, & \\
\Delta(e_{3})&=&0, & \\
\Delta(e_{4})&=&c_{2,4}e_{2}+c_{3,4}e_{3},\\
\Delta(e_{5})&=&c_{2,5}e_{2}+c_{3,5}e_{3},\\
\end{array}$$

\item for the algebra $\lambda_{5}:$
$$\begin{array}{llll}
\Delta(e_{1})&=&c_{1,1}e_{1}+c_{2,1}e_{2}+c_{3,1}e_{3},\\
\Delta(e_{2})&=&c_{2,2}e_{2}+c_{3,2}e_{3}, & \\
\Delta(e_{3})&=&c_{3,3}e_{3}, & \\
\Delta(e_{4})&=&c_{3,4}e_{3}+c_{4,4}e_{4}+c_{5,4}e_{5},\\
\Delta(e_{5})&=&c_{3,5}e_{3}+c_{4,5}e_{4}+c_{5,5}e_{5},\\
\end{array}$$

\item for the algebra $\lambda_{6}^{\alpha}:$
$$\begin{array}{llll}
\Delta(e_{1})&=&c_{1,1}e_{1}+c_{2,1}e_{2}+c_{3,1}e_{3},\\
\Delta(e_{2})&=&c_{2,2}e_{2}+c_{3,2}e_{3}, & \\
\Delta(e_{3})&=&c_{3,3}e_{3}, & \\
\Delta(e_{4})&=&c_{3,4}e_{3}+c_{4,4}e_{4}+c_{5,4}e_{5},\\
\Delta(e_{5})&=&c_{3,5}e_{3}+\alpha c_{4,5}e_{4}+c_{5,5}e_{5},\\
\end{array}$$

\item for the algebra $\mu_{1}:$
$$\begin{array}{llll}
\Delta(e_{1})&=&c_{1,1}e_1+c_{2,1}e_{2}+c_{3,1}e_{3}+c_{4,1}e_{4}+c_{5,1}e_{5},\\
\Delta(e_{2})&=&c_{2,2}e_2+c_{3,2}e_3+c_{5,2}e_{5}, & \\
\Delta(e_{3})&=&c_{3,3}e_3, & \\
\Delta(e_{4})&=&c_{3,4}e_{3}+c_{4,4}e_{4}+c_{5,4}e_5,\\
\Delta(e_{5})&=&c_{5,5}e_{5},\\
\end{array}$$

\item for the algebra $\mu_{2}:$
$$\begin{array}{llll}
\Delta(e_{1})&=&c_{1,1}e_1+c_{2,1}e_{2}+c_{3,1}e_{3}+c_{4,1}e_{4}+c_{5,1}e_{5},\\
\Delta(e_{2})&=&c_{2,2}e_2+c_{3,2}e_3+c_{5,2}e_{5}, & \\
\Delta(e_{3})&=&c_{3,5}e_3, & \\
\Delta(e_{4})&=&c_{2,4}e_2+c_{3,4}e_{3}+c_{4,4}e_{4}+c_{5,4}e_5,\\
\Delta(e_{5})&=&c_{5,5}e_{5},\\
\end{array}$$

\item for the algebra $\mu_{3}:$
$$\begin{array}{llll}
\Delta(e_{1})&=&c_{1,1}e_1+c_{2,1}e_{2}+c_{3,1}e_{3}+c_{4,1}e_{4}+c_{5,1}e_{5},\\
\Delta(e_{2})&=&c_{2,2}e_2+c_{3,2}e_3+c_{5,2}e_{5}, & \\
\Delta(e_{3})&=&c_{3,3}e_3, & \\
\Delta(e_{4})&=&c_{3,4}e_{3}+c_{4,4}e_{4}+c_{5,4}e_5,\\
\Delta(e_{5})&=&c_{3,5}e_{3}+c_{4,5}e_{4},\\
\end{array}$$

\item for the algebra $\mu_{4}:$
$$\begin{array}{llll}
\Delta(e_{1})&=&c_{2,1}e_{2}+c_{3,1}e_{3}+c_{4,1}e_{4}+c_{5,1}e_{5},\\
\Delta(e_{2})&=&c_{3,2}e_3+c_{5,2}e_{5}, & \\
\Delta(e_{3})&=&c_{3,3}e_3, & \\
\Delta(e_{4})&=&c_{2,4}e_2+c_{3,4}e_{3}+c_{4,4}e_{4}+c_{5,4}e_5,\\
\Delta(e_{5})&=&c_{3,5}e_{3}+c_{4,5}e_{4},\\
\end{array}$$

\item for the algebra $\mu_{5}:$
$$\begin{array}{llll}
\Delta(e_{1})&=&c_{1,1}e_{1}+c_{2,1}e_{2}+c_{3,1}e_{3}+c_{4,1}e_{4}+c_{5,1}e_{5},\\
\Delta(e_{2})&=&c_{2,2}e_{2}+c_{3,2}e_3+c_{5,2}e_{5}, & \\
\Delta(e_{3})&=&c_{3,3}e_3, & \\
\Delta(e_{4})&=&c_{2,4}e_2+c_{3,4}e_{3}+c_{4,4}e_{4}+c_{5,4}e_5,\\
\Delta(e_{5})&=&c_{3,5}e_{3}+c_{5,5}e_{5},\\
\end{array}$$

\item for the algebra $\mu_{6}:$
$$\begin{array}{llll}
\Delta(e_{1})&=&c_{2,1}e_{2}+c_{3,1}e_{3}+c_{4,1}e_{4}+c_{5,1}e_{5},\\
\Delta(e_{2})&=&c_{3,1}e_3+c_{5,2}e_{5}, & \\
\Delta(e_{3})&=&0, & \\
\Delta(e_{4})&=&c_{2,4}e_2+c_{3,4}e_{3}+c_{5,4}e_5,\\
\Delta(e_{5})&=&c_{3,3}e_{3},\\
\end{array}$$

\item for the algebra $\mu_{7}^{\alpha\neq 1}:$
$$\begin{array}{llll}
\Delta(e_{1})&=&c_{1,1}e_{1}+c_{2,1}e_{2}+c_{3,1}e_{3}+c_{4,1}e_{4}+c_{5,1}e_{5},\\
\Delta(e_{2})&=&c_{2,2}e_{2}+c_{3,2}e_3+(1+\alpha)c_{5,2}e_{5}, & \\
\Delta(e_{3})&=&c_{3,3}e_3, & \\
\Delta(e_{4})&=&c_{3,4}e_{3}+c_{4,4}e_{4}+c_{5,4}e_5,\\
\Delta(e_{5})&=&c_{5,5}{5},\\
\end{array}$$

\item for the algebra $\mu_{7}^{1}:$
$$\begin{array}{llll}
\Delta(e_{1})&=&c_{1,1}e_{1}+c_{2,1}e_{2}+c_{3,1}e_{3}+c_{4,1}e_{4}+c_{5,1}e_{5},\\
\Delta(e_{2})&=&c_{2,2}e_{2}+c_{3,2}e_3+c_{5,2}e_{5}, & \\
\Delta(e_{3})&=&c_{3,3}e_3, & \\
\Delta(e_{4})&=&c_{2,4}e_{2}+c_{3,4}e_{3}+c_{4,4}e_{4}+c_{5,4}e_5,\\
\Delta(e_{5})&=&c_{3,5}e_{3}+c_{5,5}e_{5},\\
\end{array}$$

\item for the algebra $\mu_{8}^{\alpha\neq 1}:$
$$\begin{array}{llll}
\Delta(e_{1})&=&c_{1,1}e_{1}+c_{2,1}e_{2}+c_{3,1}e_{3}+c_{4,1}e_{4}+c_{5,1}e_{5},\\
\Delta(e_{2})&=&c_{2,2}e_{2}+c_{3,2}e_3+(1+\alpha)c_{5,2}e_{5}, & \\
\Delta(e_{3})&=&c_{3,3}e_3, & \\
\Delta(e_{4})&=&c_{2,4}e_{2}+c_{3,4}e_{3}+c_{4,4}e_{4}+c_{5,4}e_5,\\
\Delta(e_{5})&=&c_{5,5}e_{5},\\
\end{array}$$

\item for the algebra $\mu_{8}^{1}:$
$$\begin{array}{llll}
\Delta(e_{1})&=&c_{1,1}e_{1}+c_{2,1}e_{2}+c_{3,1}e_{3}+c_{4,1}e_{4}+c_{5,1}e_{5},\\
\Delta(e_{2})&=&c_{2,2}e_{2}+c_{3,2}e_3+c_{5,2}e_{5}, & \\
\Delta(e_{3})&=&c_{3,3}e_3, & \\
\Delta(e_{4})&=&c_{2,4}e_{2}+c_{3,4}e_{3}+c_{4,4}e_{4}+c_{5,4}e_5,\\
\Delta(e_{5})&=&c_{3,4}e_{3}+c_{5,5}e_{5},\\
\end{array}$$

\item for the algebra $\mu_{9}:$
$$\begin{array}{llll}
\Delta(e_{1})&=&c_{1,1}e_{1}+c_{2,1}e_{2}+c_{3,1}e_{3}+c_{5,1}e_{5},\\
\Delta(e_{2})&=&c_{2,2}e_{2}+c_{3,2}e_3, & \\
\Delta(e_{3})&=&c_{3,3}e_3, & \\
\Delta(e_{4})&=&c_{3,4}e_{3}+c_{4,4}e_{4}+c_{5,4}e_5,\\
\Delta(e_{5})&=&c_{5,5}e_{5},\\
\end{array}$$

\item for the algebra $\mu_{10}:$
$$\begin{array}{llll}
\Delta(e_{1})&=&c_{1,1}e_{1}+c_{2,1}e_{2}+c_{3,1}e_{3}+c_{5,1}e_{5},\\
\Delta(e_{2})&=&c_{2,2}e_{2}+c_{3,2}e_3, & \\
\Delta(e_{3})&=&c_{3,3}e_3, & \\
\Delta(e_{4})&=&c_{3,4}e_{3}+c_{4,4}e_{4}+c_{5,4}e_5,\\
\Delta(e_{5})&=&c_{5,5}e_{5},\\
\end{array}$$

\item for the algebra $\mu_{11}:$
$$\begin{array}{llll}
\Delta(e_{1})&=&c_{2,1}e_{2}+c_{3,1}e_{3}+c_{5,1}e_{5},\\
\Delta(e_{2})&=&c_{2,2}e_3, & \\
\Delta(e_{3})&=&0, & \\
\Delta(e_{4})&=&c_{3,4}e_{3}+c_{5,4}e_5,\\
\Delta(e_{5})&=&0,\\
\end{array}$$

\item for the algebra $\mu_{12}:$
$$\begin{array}{llll}
\Delta(e_{1})&=&c_{1,1}e_{1}+c_{2,1}e_{2}+c_{3,1}e_{3}+c_{4,1}e_{4}+c_{5,1}e_{5},\\
\Delta(e_{2})&=&c_{2,2}e_{2}+c_{3,2}e_3, & \\
\Delta(e_{3})&=&c_{3,3}e_3, & \\
\Delta(e_{4})&=&c_{2,4}e_{2}+c_{3,4}e_{3}+c_{4,4}e_{4}+c_{5,4}e_5,\\
\Delta(e_{5})&=&c_{3,5}e_{3}+c_{5,5}e_{5},\\
\end{array}$$

\item for the algebra $\mu_{13}:$
$$\begin{array}{llll}
\Delta(e_{1})&=&c_{1,1}e_{1}+c_{2,1}e_{2}+c_{3,1}e_{3}+c_{4,1}e_{4}+c_{5,1}e_{5},\\
\Delta(e_{2})&=&c_{2,2}e_{2}+c_{3,2}e_3, & \\
\Delta(e_{3})&=&c_{3,3}e_3, & \\
\Delta(e_{4})&=&c_{2,4}e_{2}+c_{3,4}e_{3}+c_{4,4}e_{4}+c_{5,4}e_5,\\
\Delta(e_{5})&=&c_{3,5}e_{3}+c_{5,5}e_{5},\\
\end{array}$$

\item for the algebra $\mu_{14}:$
$$\begin{array}{llll}
\Delta(e_{1})&=&c_{1,1}e_{1}+c_{2,1}e_{2}+c_{3,1}e_{3}+c_{5,1}e_{5},\\
\Delta(e_{2})&=&c_{2,2}e_{2}+c_{3,2}e_3, & \\
\Delta(e_{3})&=&c_{3,3}e_3, & \\
\Delta(e_{4})&=&c_{2,4}e_{2}+c_{3,4}e_{3}+c_{4,4}e_{4}+c_{5,4}e_5,\\
\Delta(e_{5})&=&c_{3,5}e_{3}+c_{5,5}e_{5},\\
\end{array}$$

\item for the algebra $\mu_{15}:$
$$\begin{array}{llll}
\Delta(e_{1})&=&c_{2,1}e_{2}+c_{3,1}e_{3}+c_{5,1}e_{5},\\
\Delta(e_{2})&=&c_{3,2}e_3, & \\
\Delta(e_{3})&=&0, & \\
\Delta(e_{4})&=&c_{2,4}e_{2}+c_{3,4}e_{3}+c_{5,4}e_5,\\
\Delta(e_{5})&=&c_{3,5}e_{3},\\
\end{array}$$

\item for the algebra $\mu_{16}:$
$$\begin{array}{llll}
\Delta(e_{1})&=&c_{1,1}e_{1}+c_{2,1}e_{2}+c_{3,1}e_{3}+c_{5,1}e_{5},\\
\Delta(e_{2})&=&c_{2,2}e_{2}+c_{3,2}e_3, & \\
\Delta(e_{3})&=&c_{3,3}e_3, & \\
\Delta(e_{4})&=&c_{2,4}e_{2}+c_{3,4}e_{3}+c_{4,4}e_{4}+c_{5,4}e_5,\\
\Delta(e_{5})&=&c_{3,5}e_{3}+c_{5,5}e_{5},\\
\end{array}$$

\item for the algebra $\mu_{17}:$
$$\begin{array}{llll}
\Delta(e_{1})&=&c_{2,1}e_{2}+c_{3,1}e_{3}+c_{5,1}e_{5},\\
\Delta(e_{2})&=&c_{3,2}e_3, & \\
\Delta(e_{3})&=&0, & \\
\Delta(e_{4})&=&c_{2,4}e_{2}+c_{3,4}e_{3}+c_{5,4}e_5,\\
\Delta(e_{5})&=&c_{3,5}e_{3},\\
\end{array}$$

\item for the algebra $\mu_{18}:$
$$\begin{array}{llll}
\Delta(e_{1})&=&c_{1,1}e_{1}+c_{3,1}e_{3}+c_{5,1}e_{5},\\
\Delta(e_{2})&=&c_{2,2}e_{2}, & \\
\Delta(e_{3})&=&c_{3,3}e_3, & \\
\Delta(e_{4})&=&c_{3,4}e_{3}+c_{4,4}e_{4}+c_{5,4}e_5,\\
\Delta(e_{5})&=&c_{3,5}e_{3}+c_{5,5}e_{5},\\
\end{array}$$

\item for the algebra $\mu_{19}:$
$$\begin{array}{llll}
\Delta(e_{1})&=&c_{1,1}e_{1}+c_{2,1}e_{2}+c_{3,1}e_{3}+c_{4,1}e_{4}+c_{5,1}e_{5},\\
\Delta(e_{2})&=&c_{2,2}e_{2}+c_{3,2}e_3, & \\
\Delta(e_{3})&=&c_{3,3}e_3, & \\
\Delta(e_{4})&=&c_{2,4}e_{2}+c_{3,4}e_{3}+c_{5,4}e_5,\\
\Delta(e_{5})&=&c_{5,3}e_{3},\\
\end{array}$$

\item for the algebra $\mu_{20}:$
$$\begin{array}{llll}
\Delta(e_{1})&=&c_{1,1}e_{1}+c_{2,1}e_{2}+c_{3,1}e_{3}+c_{5,1}e_{5},\\
\Delta(e_{2})&=&c_{2,2}e_{2}+c_{2,3}e_3, & \\
\Delta(e_{3})&=&c_{3,3}e_3, & \\
\Delta(e_{4})&=&c_{2,4}e_{2}+c_{3,4}e_{3}+c_{4,4}e_{4}+c_{5,4}e_5,\\
\Delta(e_{5})&=&c_{3,5}e_{3}+c_{5,5}e_{5},\\
\end{array}$$

\item for the algebra $\mu_{21}^{\alpha},\ \alpha\in \{\pm\imath\}:$
$$\begin{array}{llll}
\Delta(e_{1})&=&c_{2,1}e_{2}+c_{3,1}e_{3}+c_{5,1}e_{5},\\
\Delta(e_{2})&=&c_{3,2}e_3, & \\
\Delta(e_{3})&=&0, & \\
\Delta(e_{4})&=&c_{2,4}e_2+c_{3,4}e_{3}+c_{5,4}e_5,\\
\Delta(e_{5})&=&c_{3,5}e_3,\\
\end{array}$$

\item for the algebra $\mu_{22}^{1},:$
$$\begin{array}{llll}
\Delta(e_{1})&=&c_{1,1}e_{1}+c_{2,1}e_{2}+c_{3,1}e_{3}+c_{4,1}e_{4}+c_{5,1}e_{5},\\
\Delta(e_{2})&=&c_{2,2}e_{2}+c_{3,2}e_3+c_{5,2}e_{5}, & \\
\Delta(e_{3})&=&c_{3,3}e_3, & \\
\Delta(e_{4})&=&c_{1,4}e_1+c_{2,4}e_2+c_{3,4}e_{3}+c_{4,4}e_{4}+c_{5,4}e_5,\\
\Delta(e_{5})&=&c_{2,5}e_2+c_{3,5}e_{3}+c_{5,5}e_{5},\\
\end{array}$$

\item for the algebra $\mu_{22}^{0},:$
$$\begin{array}{llll}
\Delta(e_{1})&=&c_{1,1}e_{1}+c_{2,1}e_{2}+c_{3,1}e_{3}+c_{4,1}e_{4}+c_{5,1}e_{5},\\
\Delta(e_{2})&=&c_{2,2}e_{2}+c_{3,2}e_{3}+c_{5,2}e_{5}, & \\
\Delta(e_{3})&=&c_{3,3}e_3, & \\
\Delta(e_{4})&=&c_{2,4}e_2+c_{3,4}e_{3}+c_{4,4}e_{4}+c_{5,5}e_5,\\
\Delta(e_{5})&=&c_{5,3}e_{3}+c_{5,5}e_{5},\\
\end{array}$$

\item for the algebra $\mu_{22}^{\alpha},\ \alpha\neq0;1:$
$$\begin{array}{llll}
\Delta(e_{1})&=&c_{1,1}e_{1}+c_{2,1}e_{2}+c_{3,1}e_{3}+c_{4,1}e_{4}+c_{5,1}e_{5},\\
\Delta(e_{2})&=&c_{2,2}e_{2}+(1+\alpha)c_{3,2}e_{3}+(1+\alpha)c_{5,2}e_{5}, & \\
\Delta(e_{3})&=&c_{3,3}e_3, & \\
\Delta(e_{4})&=&(\alpha^2+\alpha)c_{4,1}e_1+c_{2,4}e_2+c_{3,4}e_{3}+c_{4,4}e_{4}+c_{5,4}e_5,\\
\Delta(e_{5})&=&c_{2,5}e_2+c_{3,5}e_{3}+c_{5,5}e_{5}.\\
\end{array}$$
\end{itemize}
\end{thm}
\begin{proof} The proof follows by straightforward calculations similarly to
the proof of Theorem \ref{small3}.
\end{proof}

By the direct  calculation we obtain the dimensions of the spaces of derivation and
local derivations of $5$-dimensional complex  naturally graded quasi-filiform associative algebras.
$$\begin{tabular}{|p{2.5cm}|p{4cm}|p{4cm}|}
    \hline
  Algebra &  The dimensions of the space of derivations & The dimensions of the space of local derivations \\
  \hline
    $\lambda_1$  & $8$& $16$ \\[1mm]
  \hline
    $\lambda_2$  & $7$& $14$ \\[1mm]
  \hline
    $\lambda_3$   & $6$& $13$ \\[1mm]
  \hline
    $\lambda_4$   & $6$& $9$ \\[1mm]
  \hline
    $\lambda_5$   & $8$& $12$ \\[1mm]
  \hline
  $\lambda_6^{\alpha}$   & $6$& $12$ \\[1mm]
  \hline
   $\mu_1$   & $8$& $13$ \\[1mm]
  \hline
   $\mu_2$   & $7$& $14$ \\[1mm]
  \hline
   $\mu_3$   & $7$& $14$ \\[1mm]
  \hline
   $\mu_4$   & $6$& $13$ \\[1mm]
  \hline
   $\mu_5$   & $8$& $15$ \\[1mm]
  \hline
   $\mu_6$   & $7$& $10$ \\[1mm]
  \hline
   $\mu_7^{\alpha\neq1}$   & $8$& $13$ \\[1mm]
  \hline

   $\mu_7^1$   & $9$& $15$ \\[1mm]
  \hline
   $\mu_8^{\alpha\neq1}$   & $7$& $14$ \\[1mm]
  \hline
   $\mu_8^1$   & $7$& $15$ \\[1mm]
  \hline
   $\mu_9$   & $6$& $11$ \\[1mm]
  \hline
   $\mu_{10}$   & $6$& $11$ \\[1mm]
  \hline
   $\mu_{11}$   & $5$& $6$ \\[1mm]
  \hline
   $\mu_{12}$   & $7$& $14$ \\[1mm]
  \hline
   $\mu_{13}$   & $7$& $14$ \\[1mm]
  \hline
   $\mu_{14}$   & $6$& $13$ \\[1mm]
  \hline
   $\mu_{15}$   & $5$& $8$ \\[1mm]
  \hline
   $\mu_{16}$   & $6$& $13$ \\[1mm]
  \hline
   $\mu_{17}$   & $5$& $8$ \\[1mm]
  \hline
   $\mu_{18}$   & $5$& $10$ \\[1mm]
  \hline
   $\mu_{19}$   & $6$& $12$ \\[1mm]
  \hline
   $\mu_{20}$   & $4$& $13$ \\[1mm]
  \hline
   $\mu_{21}^{\alpha},\alpha\in\{\pm i\} $   & $5$& $8$ \\[1mm]
  \hline
   $\mu_{22}^1$   & $7$& $17$ \\[1mm]
  \hline
   $\mu_{22}^0$   & $7$& $15$ \\[1mm]
  \hline
   $\mu_{22}^{\alpha\neq0;1}$   & $6$& $17$ \\[1mm]
  \hline
    \end{tabular}$$

\begin{cor} Every $5$-dimensional complex  naturally graded quasi-filiform associative algebra admits a local derivations which
are not derivations.
\end{cor}

Concerning 2-local derivations on the above algebras we have a similar result.

\begin{thm}
Each $5$-dimensional complex  naturally graded quasi-filiform associative algebra admits  $2$-local derivations which are not derivations.
\end{thm}

\begin{proof} The example of such a 2-local derivation  is similar to that given in  the proof of Theorem \ref{small31}.
\end{proof}

\section{Local and 2-Local derivations on null-filiform and filiform associative algebras}

In this section, we consider local and 2-local derivations of null-filiform ($n>2$), filiform ($n>3$), and quasi-filiform ($n>5$) algebras. All these algebras admit a local derivation which is not a derivation. At the same time we show that every  2-local derivation of a null-filiform algebra is a derivation. An example of a 2-local derivation which is not a derivation is given for filiform and quasi-filiform algebras.

\subsection{Local derivations on null-filiform and filiform associative algebras}
Now we study local derivations on $n$-dimensional null-filiform, $n$-dimensional ($n>3$) complex filiform and $n$-dimensional $(n \geq 6)$ complex naturally graded quasi-filiform  non-split associative algebras.

\begin{thm}\label{sec1} Let $\Delta$ be a linear map from a null-filiform associative algebras into itself.
Then $\Delta$ is a local derivation, if and only if:

\begin{equation*}\begin{split}
\Delta(e_i)&=\sum\limits_{j=i}^{n}c_{j,i}e_j,\ \ 1\leq i\leq n.
\end{split}
\end{equation*}
 \end{thm}

\begin{proof}
Let $\mathfrak{C}$ be the matrix of a local derivation $\Delta$  on $\mu_{0}^n:$

\[
\mathfrak{C}=\left(
              \begin{array}{ccccccc}
                c_{1,1}   & c_{1,2}        & \cdots & c_{1,n-1}   & c_{1,n}       \\
                c_{2,1}   & c_{2,2}        & \cdots & c_{2,n-1}   & c_{2,n}       \\
                c_{3,1}   & c_{3,2}        & \cdots & c_{3,n-1}   & c_{3,n}       \\
                \cdots    & \cdots         & \cdots & \cdots      & \cdots        \\
                c_{n-1,1} & c_{n-1,2}      & \cdots & c_{n-1,n-1} & c_{n-1,n}      \\
                c_{n,1}   & c_{n,2}        & \cdots & c_{n,n-1}   & c_{n,n}       \\
                \end{array}
            \right)
\]

By the definition for every $x=\sum\limits_{i=1}^nx_ie_i\in\mu_{0}^n$
there exists a derivation $D_x$ on $\mu_{0}^n$ such that
$$
\Delta(x)=D_x(x).
$$
By Proposition \ref{prop}, the derivarion $D_x$ has the following matrix form:

\[
\mathfrak{C}_x=\left(
              \begin{array}{cccccccc}
 \alpha_1^x    & 0              & 0               &\cdots   & 0                &  0                     \\
 \alpha_2^x    & 2\alpha_1^x    & 0               & \cdots  & 0                &  0                       \\
 \alpha_3^x    & 2\alpha_2^x    & 3\alpha_1^x     & \cdots  & 0                &  0                       \\
 \cdots        & \cdots         & \cdots          & \cdots  & \cdots           &  \cdots                  \\
 \alpha_{n-2}^x& 2\alpha_{n-3}^x& 3\alpha_{n-4}^x & \cdots  & 0                &  0                      \\
 \alpha_{n-1}^x& 2\alpha_{n-2}^x& 3\alpha_{n-3}^x & \cdots  & (n-1)\alpha_{1}^x&  0                   \\
 \alpha_n^x    & 2\alpha_{n-1}^x& 3\alpha_{n-2}^x & \cdots  & (n-1)\alpha_{2}^x& n\alpha_{1}^x                      \\
 \end{array}
            \right)
\]

 For the matrix $\mathfrak{C}$  of $\Delta$ by choosing subsequently $x=e_1,...,x=e_n$ and using $\Delta(x)=D_x(x),$ i.e. $\mathfrak{C}\overline{x}=D_x(\overline{x}),$
where $\overline{x}$ is the vector corresponding to $x$, we obtain

\[
\mathfrak{C}=\left(
              \begin{array}{ccccccc}
                c_{1,1}   & 0              & \cdots & 0   & 0      \\
                c_{2,1}   & c_{2,2}        & \cdots & 0   & 0      \\
                c_{3,1}   & c_{3,2}        & \cdots & 0   & 0       \\
                \cdots    & \cdots         & \cdots & \cdots      & \cdots        \\
                c_{n-1,1} & c_{n-1,2}      & \cdots & 0   & 0      \\
                c_{n-1,1} & c_{n-1,2}      & \cdots & c_{n-1,n-1} & 0      \\
                c_{n,1}   & c_{n,2}        & \cdots & c_{n,n-1}   & c_{n,n}       \\
                \end{array}
            \right)
\]

 Conversely, suppose that the matrix of $\Delta$ has the above form and let us show that $\Delta$ is a local derivation. For each $x$ we must find a derivation $D_x$ such that  $\Delta(x)=D_x(x).$  We have $\mathfrak{C}\overline{x}=\mathfrak{C}_x(\overline{x}),$ where $\overline{x}$ is the vector corresponding to $x=\sum\limits_{i=1}^nx_ie_i,$. This implies  the following system of equalities
{\small
$$  \begin{array}{ll}
    c_{1,1}x_1 &=\alpha_{1}^xx_1, \\
    c_{2,1}x_1+c_{2,2}x_2 &= \alpha_2^xx_1+2\alpha_1^xx_2,\\
    \cdots &  \cdots\\
    c_{n-1,1}x_1+c_{n-1,2}x_2+\cdots+c_{n-1,n-1}x_{n-1}&=\alpha_{n-1}^xx_1+2\alpha_{n-2}^xx_2+\cdots+(n-1)\alpha_{1}^xx_{n-1},\\
    c_{n,1}x_{1}+c_{n,2}x_2+\cdots+c_{n,n}x_n &=\alpha_n^xx_1+2\alpha_{n-1}^xx_2+\cdots+n\alpha_1^xx_{n}, \\
  \end{array}
$$}

Let us consider  two cases separately:

\textbf{Case 1:} If $x_1\neq 0,$ then $\alpha_1^x=c_{1,1},$

\begin{equation*}\begin{split}
\alpha_{i}^x&=c_{i,1}+\frac{1}{x_1}\left(\sum\limits_{j=2}^{i}(c_{i,j}-j\alpha_{i-j+1}^x)x_j\right),
\quad 2\leq i \leq n.\\
\end{split}\end{equation*}

\textbf{Case 2:} If $x_1=x_2=...=x_{t-1}=0$ and $x_t\neq0,$ then $\alpha_{1}^x=\frac{1}{t}c_{t,t},$ $2\leq t\leq n$
\begin{equation*}\begin{split}
\alpha_{i-t+1}^x&=\frac{c_{i,t}}{t}+\frac{1}{tx_t}\left(\sum\limits_{j=t+1}^{i}(c_{i,j}-j\alpha_{i-j+1}^x)x_j\right)
\quad t+1\leq i \leq n.
\end{split}\end{equation*}

We find $\alpha_i^x,\ 1\leq i\leq n$. Thus, the matrix of the derivation $D_x$ is determined such that $\Delta(x)=D_x(x).$ The proof is complete.

\end{proof}

\begin{thm}\label{sec1.2.2} Let $\Delta$ be a linear map on an $n$-dimensional ($n>3$) complex filiform associative algebra.
Then $\Delta$ is a local derivation, if and only if:

\begin{itemize}
  \item for the algebras $\mu_{1,1}^n,\ \mu_{1,3}^n$
 \begin{equation}\label{Loc1} LocDer: \left\{\begin{array}{ll}
 \Delta(e_1)&=\sum\limits_{j=1}^{n}c_{j,1}e_j,\\[1mm]
 \Delta(e_i)&=\sum\limits_{j=i}^{n-1}c_{j,i}e_j,\ \ 2\leq i\leq n-1\\[1mm]
  \Delta(e_n)&=c_{n-1,n}e_{n-1}+c_{n,n}e_n. \\[1mm]
 \end{array}\right.
 \end{equation}
  \item for the algebras $\mu_{1,2}^n,\ \mu_{1,4}^n$
 \begin{equation}\label{Loc1} LocDer: \left\{\begin{array}{ll}
 \Delta(e_1)&=\sum\limits_{j=1}^{n}c_{j,1}e_j,\\[1mm]
 \Delta(e_i)&=\sum\limits_{j=i}^{n-1}c_{j,i}e_j,\ \ 2\leq i\leq n-1\\[1mm]
  \Delta(e_n)&=c_{n-2,n}e_{n-2}+c_{n-1,n}e_{n-1}+c_{n,n}e_n. \\[1mm]
 \end{array}\right.
 \end{equation}
\end{itemize}
\end{thm}

\begin{proof} We shall prove the theorem for the algebra $\mu_{1,1}^n;$  for the algebras $\mu_{1,2}^n,\ \mu_{1,3}^n,\ \mu_{1,4}^n$  the proofs are similar.

Let $\mathfrak{C}$ be the matrix of a local derivation $\Delta$  on $\mu_{1,1}^n:$

\[
\mathfrak{C}=\left(
              \begin{array}{ccccccc}
                c_{1,1}   & c_{1,2}        & \cdots & c_{1,n-1}   & c_{1,n}       \\
                c_{2,1}   & c_{2,2}        & \cdots & c_{2,n-1}   & c_{2,n}       \\
                c_{3,1}   & c_{3,2}        & \cdots & c_{3,n-1}   & c_{3,n}       \\
                \cdots    & \cdots         & \cdots & \cdots      & \cdots        \\
                c_{n-1,1} & c_{n-1,2}      & \cdots & c_{n-1,n-1} & c_{n-1,n}      \\
                c_{n,1}   & c_{n,2}        & \cdots & c_{n,n-1}   & c_{n,n}       \\
                \end{array}
            \right)
\]

By the definition for each $x=\sum\limits_{i=1}^nx_ie_i\in\mu_{1,1}^n$
there exists a derivation $D_x$ on $\mu_{1,1}^n$ such that
$$
\Delta(x)=D_x(x).
$$
By Proposition \ref{prop1}, the derivation $D_x$ has the following matrix form:

\[
\mathfrak{C}_x=\left(
              \begin{array}{cccccccc}
 \alpha_1^x    & 0              & 0               &\cdots   & 0                &  0                     \\
 \alpha_2^x    & 2\alpha_1^x    & 0               & \cdots  & 0                &  0                       \\
 \alpha_3^x    & 2\alpha_2^x    & 3\alpha_1^x     & \cdots  & 0                &  0                       \\
 \cdots        & \cdots         & \cdots          & \cdots  & \cdots           &  \cdots                  \\
 \alpha_{n-2}^x& 2\alpha_{n-3}^x& 3\alpha_{n-4}^x & \cdots  & 0                &  0                      \\
 \alpha_{n-1}^x& 2\alpha_{n-2}^x& 3\alpha_{n-3}^x & \cdots  & (n-1)\alpha_{1}^x&  \beta_{n-1}^x                   \\
 \alpha_n^x    & 0              & 0               & \cdots  & 0                &  \beta_n^x                      \\
 \end{array}
            \right)
\]

Let $\mathfrak{C}$ be the matrix of $\Delta$ then by choosing subsequently $x=e_1,...,x=e_n$ and using $\Delta(x)=D_x(x),$ i.e. $\mathfrak{C}\overline{x}=D_x(\overline{x}),$
where $\overline{x}$ is the vector corresponding to $x$, we obtain

\[
\mathfrak{C}=\left(
              \begin{array}{ccccccc}
                c_{1,1}   & 0              & \cdots & 0   & 0      \\
                c_{2,1}   & c_{2,2}        & \cdots & 0   & 0      \\
                c_{3,1}   & c_{3,2}        & \cdots & 0   & 0       \\
                \cdots    & \cdots         & \cdots & \cdots      & \cdots        \\
                c_{n-1,1} & c_{n-1,2}      & \cdots & 0   & 0      \\
                c_{n-1,1} & c_{n-1,2}      & \cdots & c_{n-1,n-1} & c_{n-1,n}      \\
                c_{n,1}   & 0              & \cdots & 0   & c_{n,n}       \\
                \end{array}
            \right)
\]

 Conversely, suppose that the matrix of $\Delta$ has the above form and let us show that $\Delta$ is a local derivation, i.e. for each $x$ we must find a derivation $D_x$ such that  $\Delta(x)=D_x(x).$ We have $\mathfrak{C}\overline{x}=\mathfrak{C}_x(\overline{x}),$ where $\overline{x}$ is the vector corresponding to $x=\sum\limits_{i=1}^nx_ie_i.$ We obtain the following system of equalities
{\small
$$  \begin{array}{ll}
    c_{1,1}x_1 &=\alpha_{1}^xx_1, \\
    c_{2,1}x_1+c_{2,2}x_2 &= \alpha_2^xx_1+2\alpha_1^xx_2,\\
    \cdots &  \cdots\\
    c_{n-1,1}x_1+c_{n-1,2}x_2+\cdots+c_{n-1,n-1}x_{n-1}+c_{n-1,n}x_{n}&=\alpha_{n-1}^xx_1+2\alpha_{n-2}^xx_2+\cdots+\\
    &+(n-1)\alpha_{1}^xx_{n-1}
    +\beta_{n-1}^xx_{n},\\
    c_{n,1}x_{1}+c_{n,n}x_n &=\alpha_n^xx_1+\beta_n^xx_{n}, \\
  \end{array}
$$}

Let us consider the three cases separately:

\textbf{Case 1:} If $x_1\neq 0,$ then $\alpha_1^x=c_{1,1},$

\begin{equation*}\begin{split}
\alpha_{i}^x&=c_{i,1}+\frac{1}{x_1}\left(\sum\limits_{j=2}^{i-1}(c_{i,j}-j\alpha_{i-j+1}^x)x_j\right),
\quad 2\leq i \leq n-2,\\
\alpha_{n-1}^x&=c_{n-1,1}+\frac{1}{x_1}\left(\sum\limits_{j=2}^{n-1}(c_{n-1,j}-j\alpha_{n-j}^x)x_j+(c_{n-1,n}-\beta_{n-1}^x)x_n\right),\\
\alpha_n^x&=c_{n,1}+\frac{1}{x_1}(c_{n,n}-\beta_n^x)x_n,\\
\end{split}\end{equation*}
where $\beta_{n-1}^x$ and $\beta_{n}^x$ are  arbitrary.

\textbf{Case 2:} If $x_1=x_2=...=x_{t-1}=0$ and $x_t\neq0,$ then $\alpha_{1}^x=\frac{1}{t}c_{t,t},$ $2\leq t\leq n-1$
\begin{equation*}\begin{split}
\alpha_{i-t+1}^x&=\frac{c_{i,t}}{t}+\frac{1}{tx_t}\left(\sum\limits_{j=t+1}^{i-1}(c_{i,j}-j\alpha_{i-j+1}^x)x_j\right)
\quad t+1\leq i \leq n-2,\\
\alpha_{n-t}^x&=\frac{c_{n-t,1}}{t}+\frac{1}{tx_1}\left(\sum\limits_{j=t+1}^{n-1}(c_{n-t,j}-j\alpha_{n-j}^x)x_j+
(c_{n-1,n}-\beta_{n-1}^x)x_n\right),\\
\end{split}\end{equation*}
where $\beta_{n-1}^x$ and $\beta_{n}^x$ are arbitrary.

\textbf{Case 3:} If $x_{n}\neq 0,$ then
$\beta_{n-1}^x=c_{n-1,n}$ and $\beta_{n}^x=c_{n,n}.$

We find $\alpha_i^x,\ 1\leq i\leq n,$ $\beta_{n-1}^x$ and $\beta_n^x.$

So, the matrix of the  derivation $D_x$ is defined such that  $\Delta(x)=D_x(x)$. The proof is complete.
\end{proof}

\begin{thm}\label{thm2} Each   local derivation $\Delta$ on an $n$-dimensional $(n \geq 6)$ complex naturally graded quasi-filiform  non-split associative algebra has the following form:

\begin{itemize}
  \item for the algebra $\mu_{2,1}^n,\ $
 \begin{equation*}\label{Loc1} LocDer: \left\{\begin{array}{ll}
 \Delta(e_1)&=\sum\limits_{j=1}^{n}c_{j,1}e_j,\\[1mm]
 \Delta(e_2)&=\sum\limits_{j=2}^{n-2}c_{j,2}e_j+c_{n,2}e_n,\\[1mm]
 \Delta(e_i)&=\sum\limits_{j=i}^{n-2}c_{j,i}e_j,\ \ 3\leq i\leq n-2,\\[1mm]
 \Delta(e_{n-1})&=c_{n-2,n-1}e_{n-2}+c_{n-1,n-1}e_{n-1}+c_{n,n-1}e_{n},\\[1mm]
  \Delta(e_n)&=c_{n,n}e_n. \\[1mm]
 \end{array}\right.
 \end{equation*}
  \item for the algebra $\mu_{2,2}^n(\alpha=1),$
  \begin{equation*}\label{Loc1} LocDer: \left\{\begin{array}{ll}
 \Delta(e_1)&=\sum\limits_{j=1}^{n}c_{j,1}e_j,\\[1mm]
 \Delta(e_2)&=\sum\limits_{j=2}^{n-2}c_{j,2}e_j+c_{n,2}e_n,\\[1mm]
 \Delta(e_i)&=\sum\limits_{j=i}^{n-2}c_{j,i}e_j,\ \ 3\leq i\leq n-2,\\[1mm]
 \Delta(e_{n-1})&=\sum\limits_{j=1}^{n-2}c_{j,n-1}e_{j}+c_{n,n-1}e_{n},\\[1mm]
  \Delta(e_n)&=\sum\limits_{j=2}^{n-2}c_{j,n}e_{j}+c_{n,n}e_{n}, \\[1mm]
 \end{array}\right.
 \end{equation*}
\end{itemize}
 \item for the algebra $\mu_{2,2}^n(\alpha=-1), $
 \begin{equation*}\label{Loc1} LocDer: \left\{\begin{array}{ll}
 \Delta(e_1)&=\sum\limits_{j=1}^{n}c_{j,1}e_j,\\[1mm]
 \Delta(e_2)&=\sum\limits_{j=2}^{n-2}c_{j,2}e_j+c_{n,2}e_n,\\[1mm]
 \Delta(e_i)&=\sum\limits_{j=i}^{n-2}c_{j,i}e_j,\ \ 3\leq i\leq n-2,\\[1mm]
 \Delta(e_{n-1})&=c_{n-2,n-1}e_{n-2}+c_{n,n-1}e_{n},\\[1mm]
  \Delta(e_n)&=c_{n,n}e_n. \\[1mm]
 \end{array}\right.
 \end{equation*}
  \item for the algebra $\mu_{2,2}^n(\alpha\neq-1;1), $
 \begin{equation*}\label{Loc1} LocDer: \left\{\begin{array}{ll}
 \Delta(e_1)&=\sum\limits_{j=1}^{n}c_{j,1}e_j,\\[1mm]
 \Delta(e_i)&=\sum\limits_{j=i}^{n-2}c_{j,i}e_j,\ \ 2\leq i\leq n-2,\\[1mm]
 \Delta(e_{n-1})&=c_{n-2,n-1}e_{n-2}+c_{n-1,n-1}e_{n-1}+c_{n,n-1}e_{n},\\[1mm]
  \Delta(e_n)&=c_{n,n}e_n. \\[1mm]
 \end{array}\right.
 \end{equation*}
   \item for the algebras $\mu_{2,3}^n,\mu_{2,4}^n $
 \begin{equation*}\label{Loc1} LocDer: \left\{\begin{array}{ll}
 \Delta(e_1)&=\sum\limits_{j=1}^{n-2}c_{j,1}e_j+c_{n,1}e_n,\\[1mm]
 \Delta(e_i)&=\sum\limits_{j=i}^{n-2}c_{j,i}e_j,\ \ 2\leq i\leq n-2,\\[1mm]
 \Delta(e_{n-1})&=c_{n-2,n-1}e_{n-2}+c_{n-1,n-1}e_{n-1}+c_{n,n-1}e_{n},\\[1mm]
  \Delta(e_n)&=c_{n,n}e_n. \\[1mm]
 \end{array}\right.
 \end{equation*}
\end{thm}

\begin{proof} The proof is obtained by straightforward calculations similarly to
the proof of Theorem \ref{sec1}.
\end{proof}
As above,  one can calculate  dimensions of the spaces of derivation and
local derivations of null-filiform, filiform and naturally graded quasi-filiform  non-split associative algebras.

$$\begin{tabular}{|p{2.5cm}|p{4cm}|p{4cm}|}
    \hline
  Algebra &  The dimensions of the space of derivations & The dimensions of the space of local derivations \\
  \hline
    $\mu_{0}^n$  & $n$& $\frac{n^2+n}{2}$ \\[1mm]
  \hline
    $\mu_{1,1}^n$  & $n+2$& $\frac{n^2-n+6}{2}$ \\[1mm]
  \hline
    $\mu_{1,2}^n$   & $n+1$& $\frac{n^2-n+8}{2}$ \\[1mm]
  \hline
    $\mu_{1,3}^n$   & $n+1$& $\frac{n^2-n+6}{2}$ \\[1mm]
  \hline
    $\mu_{1,4}^n$   & $n$& $\frac{n^2-n+8}{2}$ \\[1mm]
  \hline
    $\mu_{2,1}^n$   & $n+3$& $\frac{n^2-3n+14}{2}$ \\[1mm]
  \hline
  $\mu_{2,2}^n,\ \alpha=1$   & $2n-2$& $\frac{n^2-n+8}{2}$ \\[1mm]
  \hline
   $\mu_{2,2}^n,\ \alpha=-1$   & $n+3$& $\frac{n^2-3n+14}{2}$ \\[1mm]
  \hline
   $\mu_{2,2}^n,\ \alpha\neq-1;1$   & $n+2$& $\frac{n^2-3n+14}{2}$ \\[1mm]
  \hline
   $\mu_{2,3}^n$  & $n+2$& $\frac{n^2-3n+12}{2}$ \\[1mm]
  \hline
   $\mu_{2,4}^n$  & $n+1$& $\frac{n^2-3n+12}{2}$ \\[1mm]
  \hline

\end{tabular}$$

\begin{cor} The null-filiform associative algebras, filiform associative algebras, and naturally graded quasi-filiform  non-split associative algebras  admit local derivations which are not derivations.
\end{cor}

\subsection{2-Local derivations on null-filiform and filiform associative algebras}

Now we study 2-local derivations on $n$-dimensional null-filiform, $n$-dimensional ($n>3$) complex filiform and $n$-dimensional $(n \geq 6)$ complex naturally graded quasi-filiform  non-split associative algebras.

\begin{thm}\label{sec1.1.1}  Any 2-local derivation
on the null-filiform associative algebra $\mu_{0}^n$  is a derivation.
\end{thm}

\begin{proof}
 Consider a $2$-local derivation $\nabla$ on $\mu_{0}^n$ such that $\nabla(e_{1})=0$.

For the element $x=t_{1}e_{1}+t_{2}e_{2}+\dots+t_{n}e_{n}\in \mu_{0}^n$, there is a derivation $D_{e_{1},x}$ such that
 $$\nabla(e_{1})=D_{e_{1},x}(e_{1}),\quad \nabla(x)=D_{e_{1},x}(x).$$

Then we have
 $$0=\nabla(e_{1})=D_{e_{1,x}}(e_{1})=\sum_{i=1}^{n}\alpha_{i}e_{i}.$$

Hence $\alpha_{1}=\alpha_{2}=\dots=\alpha_{n}=0$ and therefore $D_{e_{1},x}=0$. Thus, $\nabla=0$.

Now let $\nabla$ be an arbitrary $2$-local derivation of the algebra $\mu_{0}^n$. Then there is a derivation $D$ such that  $\nabla(e_{1})=D(e_{1})$. An operator $\nabla-D$ is a $2$-local derivation and $(\nabla-D)(e_{1})=0$,  which implies $\nabla\equiv D$. Hence  $\nabla$ is a derivation.
\end{proof}

\begin{thm}\label{sec1.1.2}
The $n$-dimensional ($n>3$) complex filiform associative and $n$-dimensional $(n \geq 6)$ complex naturally graded quasi-filiform  non-split associative algebras admit $2$-local derivations which are not derivations.
\end{thm}

\begin{proof} We proof the theorem for the algebra $\mu_{1,1}^n;$  for the algebras $\mu_{1,2}^n,\ \mu_{1,3}^n,\ \mu_{1,4}^n,$ $\mu_{2,1}^n,\ \mu_{2,2}^n(\alpha),\ \mu_{2,3}^n,\ \mu_{2,4}^n $  the proofs are similar.
Let us define a homogeneous non additive function $f$ on $\mathbb{C}^{2}$ as follows

\[ f(z_{1},z_{n}) = \begin{cases}
\frac{z^{2}_{1}}{z_{n}}, & \text{if $z_{n}\neq0$,}\\
0, & \text{if $z_{n}=0$}.
\end{cases} \]
where $(z_1,z_n)\in\mathbb{C}.$
Consider the map $\nabla: \mu_{1,1}^n\rightarrow \mu_{1,1}^n$ defined by the rule
\begin{equation*}\label{for1.1.1}
\nabla(x)=f(x_{1},x_{n})e_{n}, \quad
\textmd{where} \quad    x=\sum_{i=1}^{n}x_{i}e_{i}\in \mu_{1,1}^n.
\end{equation*}
Since $f$ is not additive, $\nabla$ is not a derivation.

Let us show that $\nabla$ is a $2$-local derivation.
For the elements
$$x=\sum_{i=1}^{n}x_{i}e_{i}, \quad
y=\sum_{i=1}^{n}y_{i}e_{i},
$$
we search a derivation $D$ in the form:
\begin{equation*}
D(e_1)=\alpha_ne_n,\quad D(e_n)=\beta_n e_n,\quad  D(e_i)=0,\quad  3\leq i\leq n.
\end{equation*}

Assume that $\nabla(x)=D(x)$ and $\nabla(y)=D(y)$. Then
we obtain the following system of equations for $\alpha_n$ and
$\beta:$
\begin{equation}\label{for1.1.2}
\begin{cases}
x_1 \alpha_n+x_n\beta_n =f(x_1, x_n), \\
y_1\alpha_n+y_n\beta_n = f(y_1, y_n).
\end{cases}
\end{equation}

\textbf{Case 1}. Let $x_{1}y_{n}-x_{n}y_{1}=0$, then the system has infinitely many solutions, because of the right-hand side of this system is homogeneous.

\textbf{Case 2}. Let $x_{1}y_{n}-x_{n}y_{1}\neq0$, then the system has a unique solution.
The proof is complete.
\end{proof}

\end{document}